\documentclass[pdflatex,sn-mathphys-num]{sn-jnl}

\usepackage{graphicx}%
\usepackage{multirow}%
\usepackage{amsmath,amssymb,amsfonts}%
\usepackage{amsthm}%
\usepackage{mathrsfs}%
\usepackage[title]{appendix}%
\usepackage{xcolor}%
\usepackage{textcomp}%
\usepackage{manyfoot}%
\usepackage{booktabs}%
\usepackage{algorithm}%
\usepackage{algorithmic}%
\usepackage{listings}%

\usepackage[english]{babel}
\usepackage{float}
\usepackage{graphicx}
\usepackage{subcaption}
\usepackage{enumerate}
\usepackage{booktabs}
\usepackage{dsfont}
\usepackage{comment}
\usepackage{microtype}

\usepackage{todonotes}
\usepackage{tablefootnote}

\allowdisplaybreaks[4]

\numberwithin{equation}{section}
\theoremstyle{thmstyleone}%
\newtheorem{theorem}{Theorem}[section]

\newtheorem{lemma}{Lemma}[section]
\newtheorem{assumption}{Assumption}[section]
\newtheorem{proposition}{Proposition}[section]
\newtheorem{condition}{Condition}[section]

\theoremstyle{thmstyletwo}%
\newtheorem{example}{Example}[section]
\newtheorem{remark}{Remark}[section]

\theoremstyle{thmstylethree}%
\newtheorem{definition}{Definition}%

\def\Argmin{\mathop{\mathrm{Argmin}}}

\def\dom{\mathop{\mathrm{dom}}}
\def\cl{\mathop{\mathrm{cl}}}

\begin{document}

\title[Single-loop SDCAM]{Complexity and convergence analysis of a single-loop SDCAM for Lipschitz composite optimization and beyond}


\author[1]{Hao Zhang}\email{haaoo.zhang@connect.polyu.hk}

\author[2]{Naoki Marumo}\email{marumo@mist.i.u-tokyo.ac.jp}

\author*[1]{Ting Kei Pong}\email{tk.pong@polyu.edu.hk}

\author[2,3]{Akiko Takeda}\email{takeda@mist.i.u-tokyo.ac.jp}

\affil[1]{Department of Applied Mathematics, the Hong Kong Polytechnic University, Hong Kong, China}

\affil[2]{Graduate School of Information Science and Technology, University of Tokyo, Tokyo, Japan}

\affil[3]{Center for Advanced Intelligence Project, RIKEN, Tokyo, Japan}

\abstract{We develop and analyze a single-loop algorithm for minimizing the sum of a Lipschitz differentiable function $f$, a prox-friendly proper closed function $g$ (with a closed domain on which $g$ is continuous) and the composition of another prox-friendly proper closed function $h$ (whose domain is closed on which $h$ is continuous) with a continuously differentiable mapping $c$ (that is Lipschitz continuous and Lipschitz differentiable on the convex closure of the domain of $g$). Such models arise naturally in many contemporary applications, where $f$ is the loss function for data misfit, and $g$ and $h$ are nonsmooth functions for inducing desirable structures in $x$ and $c(x)$. Existing single-loop algorithms mainly focus either on the case where $h$ is Lipschitz continuous or the case where $h$ is an indicator function of a closed convex set.
In this paper, we develop a single-loop algorithm for more general possibly non-Lipschitz $h$. Our algorithm is a single-loop variant of the successive difference-of-convex approximation method (SDCAM) proposed in \cite{liu2019successive}. We show that when $h$ is Lipschitz, our algorithm exhibits an iteration complexity that matches the best known complexity result for obtaining an $(\epsilon_1,\epsilon_2,0)$-stationary point. Moreover, we show that, by assuming additionally that $\dom g$ is compact, our algorithm exhibits an iteration complexity of $\tilde{\cal O}(\epsilon^{-4})$ for obtaining an $(\epsilon,\epsilon,\epsilon)$-stationary point when $h$ is merely continuous and real-valued. Furthermore, we consider a scenario where $h$ does not have full domain and establish vanishing bounds on successive changes of iterates. Finally, in all three cases mentioned above, we show that one can construct a subsequence such that any accumulation point $x^*$ satisfies $c(x^*)\in \dom h$, and if a standard constraint qualification holds at $x^*$, then $x^*$ is a stationary point.
}

\keywords{Iteration complexity, subsequential convergence, single-loop algorithm, non-Lipschitz composite functions}



\maketitle

\section{Introduction}
In this paper, we consider the following structured optimization problem
\begin{equation}\label{problem}
    \min\limits_{x \in \mathbb{R}^n}~F(x):=f(x)+ g(x)+ h(c(x)),
\end{equation}
where we assume that the objective function is proper, $g$ and $h$ are prox-friendly proper closed functions with closed domains and are continuous in their respective domains, $f: \mathbb{R}^n \rightarrow \mathbb{R}$ is Lipschitz differentiable, and $c: \mathbb{R}^n\to \mathbb{R}^m$ is continuously differentiable on $\mathbb{R}^n$ and is Lipschitz continuous and Lipschitz differentiable on the closure of the convex hull of $\dom g$; the precise assumption associated with \eqref{problem} is presented in Assumption~\ref{assumption 1} below. Model problems of this form abound in applications such as data science, machine learning, and statistics, where $f$ is typically a smooth loss function for data fidelity, $g$ and $h$ are nonsmooth functions for inducing desired structures in $x$ and $c(x)$; see, e.g., \cite{charisopoulos2021low,markovsky2008structured,parekh2015convex}.

In the special case where $c$ is a linear map, the successive difference-of-convex approximation method (SDCAM) was introduced in \cite{liu2019successive} to solve problem \eqref{problem} under Assumption~\ref{assumption 1}. The core idea of SDCAM is to iteratively approximate the objective by replacing $h$ with its Moreau envelope $e_{1/\beta_t} h$, where $\{\beta_t\}$ is a positive sequence with $\beta_t \to \infty$, leading to subproblems of the form:
\begin{equation}\label{subsubsub}
\min_{x\in \mathbb{R}^n}\ F^{[t]}(x):=f(x)+ g(x)+ e_{1/\beta_t}h(c(x)).
\end{equation}
Taking advantage of the fact that Moreau envelopes are difference-of-convex (DC) functions, these problems are solved approximately by a variant of the difference-of-convex algorithm (DCA) to generate $\{x^t\}$, whose accumulation points were shown to be stationary points of \eqref{problem} under suitable assumptions. The SDCAM was later adapted to solve problems with multiple linearly-structured rank constraints in \cite{liu2020hybrid}, and was extended to handle a class of stochastic optimization problems in \cite{metel2021simple,xu2019non}.

Note that SDCAM is an example of {\em double-loop} algorithm, with the subproblems being increasingly ill-conditioned as $\beta_t\to \infty$. For these kinds of double-loop algorithm, one often needs to carefully design the update rule of the parameters describing the subproblems (e.g., the $\{\beta_t\}$ for SDCAM) and the termination criteria of the subproblem solvers for efficient practical implementations. In contrast, {\em single-loop} algorithms do not require that the subproblems be solved to the required accuracy (or any tuning of the number of inner iterations to improve the solution accuracy), and thus are often simpler to implement in practice. The study of single-loop algorithms has received much attention recently, and many algorithms of this kind have been proposed to solve \eqref{problem} under various structural assumptions; see Table~\ref{table:smoothing_methods} for a summary of the main structural assumptions and the complexity guarantees of some recent single-loop algorithms. A notable common feature is that in most of these studies, $h$ is either Lipschitz or the indicator function of a closed convex set; see \cite{bohm2021variable,kume2024variable,yao2024note,boct2020variable,deng2025single,bot2015variable,tran2018smooth}, which assumed $h$ to be Lipschitz continuous, and \cite{Alacaoglu24a,deng2025single,tran2018smooth,LuMeiXiao24,LuXiao25}, which studied the case where $h$ is the indicator function of a closed convex set. 
However, the case where $h$ is a general non-Lipschitz function is not covered by existing single-loop algorithms.
In particular, when $h$ is the $\ell_p$ quasi-norm for some $p\in (0,1)$ (see Section~\ref{sec:nonlip} for a concrete application), direct applications of existing single-loop algorithms to the corresponding \eqref{problem} do not have convergence / complexity guarantees.

\begin{table}[h]
  \centering
  \setlength{\tabcolsep}{7pt}
  {\footnotesize
  \begin{tabular}{lllllcll}
    \toprule
    Reference & \multicolumn{4}{l}{Assumptions}  & Complexity & Remark \\\cmidrule(lr){2-5}
              & $f$ & $g$ & $h$ & $c$      &                              \\\midrule
    \cite[Theorem~3.1]{bot2015variable}                &   $0$     &cvx, lip     &cvx, lip &lin  &  $\tilde{\mathcal{O}}(\epsilon^{-1})$ & $F(x) - \inf F \le \epsilon$ \\
    \cite[Theorem~3]{tran2018smooth}                &   $0$     &cvx     &$\delta_{\{b\}}$ &lin  &  ${\mathcal{O}}(\epsilon^{-1})$ & $|F(x) - \inf F| \le \epsilon$, $\|c(x) - b\|\le \epsilon$ \\
    \cite[Theorem~4]{tran2018smooth}                &   $0$     &cvx    &cvx, lip  &lin  &  ${\mathcal{O}}(\epsilon^{-1})$ & $F(x) - \inf F \le \epsilon$ \\
    \cite[Corollary~3.1]{boct2020variable}                &   $0$     &cvx     &cvx, lip &lin  &  $\mathcal{O}(\epsilon^{-1})$ & $F(x) -\inf F \le \epsilon$ \\
    \cite[Theorem~4.2]{bohm2021variable}               & s       & $0$      & wc, lip      & lin   & $\mathcal{O}(\epsilon^{-3})$       & $(\epsilon,\epsilon,0)$-stationary point \\
    \cite[Theorem~4.2]{Alacaoglu24a}              & s       & $\delta_C$   & $\delta_{\{b\}}$      & s  & $\tilde{\mathcal{O}}(\epsilon^{-4})$       & $(\epsilon,\epsilon,0)$-stationary point\tablefootnote{\label{Tablefootnoteha}We need to point out that the references \cite{Alacaoglu24a,LuMeiXiao24,LuXiao25,deng2025single} actually considered more generally an $f$ taking the form of an expectation, and their complexity results actually give the complexity for obtaining some notion of stochastic stationary points.} \\
    \cite[Theorem~3.3]{kume2024variable}               &  s      & $\delta_D$ & wc, lip      &    s      & $\mathcal{O}(\epsilon^{-{2\alpha}/{(\alpha-1)}})$        &  ${\rm dist}(0, \partial F^{[t]}(x))\le \epsilon$\\
    \cite[Corollary~1]{LuMeiXiao24}              &   s     & $\delta_C$    & $\delta_{\{b\}}$     &   s      & $\tilde{\mathcal{O}}(\epsilon^{-\max\{\theta+2,2\theta\}})$      &  $(\epsilon, \epsilon,0)$-stationary point\footnotemark[1] \\ 
    \cite[Theorem~4.7]{yao2024note}                    & $0$   & $0$      & lip  & s           & $\mathcal{O}(\epsilon_1^{-2}\epsilon_2^{-1})$      &$(\epsilon_1, \epsilon_2,0)$-stationary point\\
    \cite[Corollary~3.2]{deng2025single}                &   s     & $\delta_{\cal M}$     &cvx, lip &s  & $\mathcal{O}(\epsilon^{-3})$         &$(\epsilon, \epsilon,0)$-stationary point\footnotemark[1]    \\
    \cite[Corollary~3.5]{deng2025single}                &   s     & $\delta_{\cal M}$     &$\delta_C$ &s  &  $\tilde{\mathcal{O}}(\epsilon^{-\max\{\theta+ 2, 2\theta\}})$ & $(\epsilon, \epsilon,0)$-stationary point\footnotemark[1] \\
    \cite[Theorem~2]{LuXiao25}                &   s, cvx     & cvx     &$\delta_{\mathbb{R}^m_-}$ & s, cvx\tablefootnote{This means each component of $c$ is convex.}  &  $\tilde{\mathcal{O}}(\epsilon^{-2})$ & $|F(x) - \inf F| \le \epsilon$, $\|[c(x)]_+\|\le \epsilon$\footnotemark[1] \\
    \bottomrule
  \end{tabular}
  }
\caption{\small Some recent single-loop algorithms for solving \eqref{problem}. Here we present the main structural assumptions on $f$, $g$ and $h$ and the known complexity results. We use s, lin, cvx,  wc, lip to denote ``Lipschitz differentiable",  ``linear", ``proper, closed and convex", ``weakly convex" and ``Lipschitz continuous", respectively; here,  $\delta_D$, $\delta_{\cal M}$, $\delta_C$ and $\delta_{\{b\}}$ are the indicator functions of a closed set $D$, a compact manifold ${\cal M}$, a closed convex set $C$, and the singleton $\{b\}$, respectively, $\alpha > 1$ is a stepsize parameter in \cite[Algorithm~1]{kume2024variable}, and $\theta\ge 1$ is a parameter in the Polyak-{\L}ojasiewicz-type (PL-type) assumption on $h$ and $c$ in \cite[Assumption~1(iv)]{LuMeiXiao24} and \cite[Assumption~4]{deng2025single}; an analogous PL-type assumption was also used in \cite[Eq.~(A5)]{Alacaoglu24a}. The complexity results are either for obtaining an $(\epsilon_1, \epsilon_2, \epsilon_3)$-stationary point (see Definition~\ref{def:ed-stationary}), or an $x$ with ${\rm dist}(0, \partial F^{[t]}(x))\le \epsilon$ (where $F^{[t]}$ is defined in \eqref{subsubsub}), or an $\epsilon$-optimal solution in the fully convex setting.}    \label{table:smoothing_methods}
\end{table}

In this paper, we develop a single-loop algorithm for instances of \eqref{problem} with possibly {\em non}-Lipschitz $h$. Our algorithm is a single-loop variant of the SDCAM, where, instead of applying one step of the SDCAM subproblem solver to the objective in \eqref{subsubsub}, we apply {\em one step} of such solver to the following function 
\[
    f(x)+ g(x)+ ({\beta_t}/{\beta_{t-1}})e_{1/ \beta_{t-1}}h(c(x)),
\]
which differs from the objective in \eqref{subsubsub} by the scaling of ${\beta_t}/{\beta_{t-1}}$ in front of $e_{1/ \beta_{t-1}}h(c(x))$.
This design leads to a pseudo-descent property (see Lemma~\ref{lemma_H} below) that underlies all convergence and complexity results in Section~\ref{convergence_section}. Under suitable choices of $\{\beta_t\}$ and a set of progressively relaxed assumptions on $h$, we discuss the corresponding global complexity of the algorithm for finding an $(\epsilon_1, \epsilon_2, \epsilon_3)$-stationary point (see Definition~\ref{def:ed-stationary}), and establish the subsequential convergence along a constructible subsequence to a stationary point under standard constraint qualifications. 

Our main results are summarized in Table~\ref{tab:contribution}. We assume $h$ to be Lipschitz in Section~\ref{sec41} and obtain an iteration complexity of $\mathcal{O}(\epsilon_1^{-2}\epsilon_2^{-1})$ for finding an $(\epsilon_1, \epsilon_2, 0)$-stationary point: this matches the best known complexity results in Table~\ref{table:smoothing_methods} for this class of problems. We then {\em go beyond} Lipschitz continuity and only require $h$ to be continuous (while also assuming the compactness of $\dom g$) in Section~\ref{sec:nonlip}. We derive an iteration complexity of $\tilde{\mathcal{O}}(\epsilon^{-4})$ for finding an $(\epsilon, \epsilon, \epsilon)$-stationary point. Finally, in Section~\ref{sec:43}, we further consider some $h$ with $\dom h \neq \mathbb{R}^m$ (see Assumption~\ref{assumption 2}). In this setting with the weakest assumptions on $h$, we still manage to establish the iteration complexity for Algorithm~\ref{alg1} to generate an $(\epsilon, \infty, \epsilon)$-stationary point:\footnote{Indeed, our result is akin to \cite[Theorem~3.3]{kume2024variable} that aims to find a point at which a suitable potential function has a ``small" subgradient; see Theorem~\ref{Thm43iteration}(i) and Remark~\ref{rem:explain_cplx}.} however, there is no explicit vanishing bound on the distance from $\{c(x^t)\}$ to $\dom h$, where $\{x^t\}$ is generated by our algorithm. Nevertheless, under each of the aforementioned settings, for the sequence $\{x^t\}$ generated by our algorithm, we show that one can construct a subsequence such that any accumulation point $x^*$ satisfies $c(x^*)\in \dom h$, and if a standard constraint qualification holds at $x^*$, then $x^*$ is a stationary point of \eqref{problem}.

\begin{table}[htbp]
    \centering
    \setlength{\tabcolsep}{6pt}
    \begin{tabular}{@{}cccccc@{}}
        \toprule
        \multicolumn{2}{c}{Assumptions} & Complexity & Remark & Subsequential \\ 
        \cmidrule(lr){1-2}
        $g$ & $h$ & & & convergence \\
        \midrule
        -- & Lipschitz & $\mathcal{O}(\epsilon_1^{-2}\epsilon_2^{-1})$ & $(\epsilon_1, \epsilon_2, 0)$-stationary point & \checkmark \\
        bounded domain & $\dom h = \mathbb{R}^m$ & $\tilde{\cal O}(\epsilon^{-4})$ & $(\epsilon, \epsilon, \epsilon)$-stationary point & \checkmark \\
        bounded domain & Assumption~\ref{assumption 2} & —— & —— & \checkmark \\
        \bottomrule
    \end{tabular}
    \caption{\small Additional assumptions on $g$ and $h$ for the complexity and convergence results in this paper, on top of Assumption~\ref{assumption 1}; see Definition~\ref{def:ed-stationary} for the definition of $(\epsilon_1, \epsilon_2, \epsilon_3)$-stationary point.}
    \label{tab:contribution}
\end{table}

The remainder of this paper is organized as follows. In Section~\ref{notation}, we review some notation and preliminary materials. In Section~\ref{alg}, we present our algorithm for \eqref{problem} and establish its well-definedness. In 
Sections~\ref{sec41}, \ref{sec:nonlip} and \ref{sec:43}, we analyze the iteration complexity and global convergence properties of our proposed algorithm under progressively relaxed assumptions on $h$. 

\section{Notation and preliminaries} \label{notation}

In this paper, we use $\mathbb{R}^n$ to denote the $n$-dimensional Euclidean space, and we use $\langle \cdot, \cdot \rangle$ and $\| \cdot\|$ to denote its inner product and the associated norm, respectively. The set of $m \times n$ matrices is denoted by $\mathbb{R}^{m \times n}$.

Let $C\subseteq \mathbb{R}^n$ be a nonempty closed set. Its indicator function $\delta_{C}$ is defined as 
\[
    \delta_{C}(x):= \begin{cases}
        0 & x\in C, \\
        \infty & x \notin C. 
    \end{cases}
\]
For a point $x\in \mathbb{R}^n$, we use ${\rm dist}(x, C)$ to denote the distance from $x$ to $C$, and ${\rm Proj}_{C}(x)$ to denote the set of projections of $x$ onto $C$; it is known that this set is a singleton if $C$ is in addition convex. For a nonempty closed and convex set $C$, its horizon cone is defined as
\[
    C^{\infty} :=  \left\{ d\in \mathbb{R}^n: x + t d \in C\ \ \forall t \ge 0\right\};
\]
here $x$ is any element in $C$ and it is known that the above definition does not depend on the choice of $x\in C$.  

A function $\varphi: \mathbb{R}^n \rightarrow (-\infty, \infty]$ is said to be proper if $\dom \varphi :=\{x: \varphi(x) < \infty\} \neq \emptyset$.
In addition, it is said to be closed if it is lower semicontinuous. Following \cite[Definition~8.3]{rockafellar2009variational}, for a proper closed function $\varphi$, we say that $v$ is a regular subgradient of $\varphi$ at $\Bar{x}\in \dom \varphi$, denoted as $v \in \widehat{\partial}\varphi(\Bar{x})$, if
\begin{align}
    \liminf\limits_{\substack{x\to \bar{x}\\ x\neq \bar{x}}} \frac{\varphi(x) -\varphi(\bar{x}) - \langle v, x - \bar{x}\rangle}{\|x -\bar{x}\|} \ge 0. \nonumber
\end{align}
Additionally, the limiting and horizon subdifferentials at $x \in \dom \varphi$ are defined as:
\begin{align}
    \partial \varphi(x) &:= \left\{ v: \exists v^t \to v,~ x^t\overset{\varphi}{\to} x~\text{with}~v^t \in \widehat{\partial} \varphi(x^t)~\text{for each}~t \right\}, \nonumber \\
    \partial^{\infty}\varphi(x) &:= \left\{ v: \exists \alpha_t\downarrow 0,~ \alpha_t v^t \to v, ~x^t \overset{\varphi}{\to}x~\text{with}~v^t \in \widehat{\partial}\varphi(x^t)~\text{for each}~t \right\}.\nonumber
\end{align}
Here, $x^t \overset{\varphi}{\rightarrow} x$ means that $\varphi(x^t) \rightarrow \varphi(x)$ and $x^t \rightarrow x$.
For any $x\in \dom \varphi$, we have the following properties:
\begin{align}
    \left\{ v: \exists v^t \to v, ~x^t\overset{\varphi}{\to}x~\text{with } v^t \in \partial \varphi(x^t)~\text{for each}~t \right\} &\subseteq \partial \varphi(x), \nonumber \\
    \left\{ v: \exists \alpha_t \downarrow 0, ~\alpha_t v^t \to v, ~x^t \overset{\varphi}{\to}x~\text{with}~v^t\in \partial \varphi(x^t)~\text{for each}~t \right\} &\subseteq \partial^{\infty}\varphi(x); \nonumber
\end{align}
see \cite[Proposition~8.7]{rockafellar2009variational}.
Recall that from \cite[Exercise~8.8(b)]{rockafellar2009variational}, the limiting subdifferential at $x$ reduces to $\{\nabla \varphi(x)\}$ if $\varphi$ is continuously differentiable at $x$. Furthermore, when $\varphi$ is proper and convex, the limiting subdifferential of $\varphi$ reduces to the classical convex subdifferential \cite[Proposition~8.12]{rockafellar2009variational}.

For a proper closed function $\varphi: \mathbb{R}^n \rightarrow (-\infty, \infty]$ with $\inf \varphi> -\infty$, its Moreau envelope for any given $\gamma >0$ is defined as 
\[
    e_{\gamma} \varphi(x):= \inf\limits_{z\in \mathbb{R}^n} \left\{ \frac{1}{2\gamma}\|x - z\|^2 + \varphi(z) \right\}.
\]
Notice that this function is defined everywhere (see \cite[Theorem 1.25]{rockafellar2009variational}).  The infimum in the definition of Moreau envelope is attained at the so-called proximal mapping of $\gamma \varphi$ at $x$, which is defined as 
\[
    {\rm Prox}_{\gamma\varphi}(x):=\Argmin\limits_{u \in \mathbb{R}^n} \left\{ \frac{1}{2\gamma}\|x -u\|^2+ \varphi(u) \right\}.
\]
For a proper closed convex function $\varphi: \mathbb{R}^n \rightarrow (-\infty, \infty]$, its horizon cone is defined as
\[
    {\rm hzn}~\varphi := \left\{ d\in \mathbb{R}^n: \varphi(x+d) \le \varphi(x)\ \ \forall \, x \in \dom \varphi \right\}.
\]
Recall from \cite[Theorem~8.7]{Rockafellar+1970} that ${\rm hzn}~\varphi$ equals the horizon cone of any nonempty level set of $\varphi$. 

For a map $G: \mathbb{R}^n \rightarrow \mathbb{R}^m$, we say that it is $\mathcal{K}$-convex for a closed convex cone $\mathcal{K}$ if
\[
    \lambda G(x) + (1 - \lambda)G(z) - G(\lambda x + (1 - \lambda)z) \in \mathcal{K} \quad\quad \forall\, x,z \in \mathbb{R}^n,\ \lambda \in [0,1].
\]
For a continuously differentiable $G:\mathbb{R}^n\to\mathbb{R}^m$, we use $J_G(x)$ to denote its Jacobian at $x$, which is the linear map defined as
\[
J_G(x)h := \lim_{t\to 0}\frac{G(x+th) - G(x)}{t}\ \ \ \forall h\in \mathbb{R}^n.
\]

We next discuss the optimality conditions for \eqref{problem}. We say that $x^*$ is a stationary point of \eqref{problem} if $x^*$ satisfies
\[
    0 \in \nabla f(x^*)+ \partial g(x^*)+ J_c(x^*)^T \partial h(c(x^*)).
\]
It can be shown that any local minimizer of \eqref{problem} satisfying constraint qualifications such as those described in Theorem~10.6 and Corollary~10.9 of \cite{rockafellar2009variational} (see \eqref{CQ} below for a precise condition) is a stationary point. We say that $x$ is an $\epsilon $-stationary point if
\begin{align}
  {\rm dist} \left(0,\nabla f(x)+\partial g(x)+ J_c(x)^T\partial h(c(x)) \right) \le \epsilon, \nonumber
\end{align}
which can be seen as a natural relaxation of the notion of stationarity. 
However, when $h$ is nonsmooth, such a point can be computationally intractable; see e.g., \cite{yao2024note,kornowski2021oracle,zhang2020complexity,tian2024no,jordan2023deterministic}.
In this paper, we adopt the following weaker notion of $(\epsilon_1, \epsilon_2,\epsilon_3)$-stationary point. We note that the notion of $(\epsilon_1,\epsilon_2,0)$-stationary point has been adopted in \cite{zhang2020complexity,jordan2023deterministic}.
\begin{definition}
  \label{def:ed-stationary}
  For $\epsilon_1$, $\epsilon_2>0$ and $\epsilon_3 \ge 0$, a point $x \in \dom g$ is called an $(\epsilon_1, \epsilon_2,\epsilon_3)$-stationary point for \eqref{problem} if there exist $y \in \mathbb{R}^m$ and $z\in \mathbb{R}^n$ such that
  \begin{align}
    {\rm dist} \left(0,\nabla f(x)+ \partial g(x)+ J_c(z)^T\partial h(y)\right) \le \epsilon_1, \quad \|c(x) -y\| \le \epsilon_2, \ \ \|x - z\|\le \epsilon_3. \nonumber
  \end{align}
\end{definition}

We end this section with the following lemma, which will be useful in Section~\ref{convergence_section} for deducing subsequential convergence from complexity bounds.

\begin{lemma} \label{choose_sub_seq}
    Let $\{\tau_k\}\subseteq \mathbb{R}_+$ satisfy $\lim_{k\rightarrow \infty} \tau_k =0$ and let $\{a_k\} \subseteq \mathbb{R}_+$ be such that $\frac{1}{T} \sum_{i=1}^T a_i \le \tau_T$ for all $T \ge 1$. Then, there exists a subsequence $\{a_{T_k}\}$ satisfying
    \[
      T_k > 1 \ \ {\rm and}\ \  a_{T_k} \le \frac{1}{T_k-1}\sum_{j=1}^{T_k-1} a_j \le \tau_{T_k-1}\ \ \ \forall k\ge 0.
    \]
\end{lemma}
\begin{proof}
    Define $b_T := \frac{1}{T} \sum_{k =1}^T a_k$. We claim that there exists a subsequence $\{b_{T_k}\}$ with $T_k > 1$ such that $b_{T_k} \le b_{T_k-1}$ for all $k$. Suppose not. Then $\{b_T\}$ is strictly increasing. Since $\{\tau_T\}$ converges to zero as $T \rightarrow \infty$ and $0 \le b_T \le \tau_T$, we have $\lim_{T\rightarrow \infty} b_T =0$, which leads to a contradiction. Therefore, there exists a subsequence $\{b_{T_k}\}$ such that $b_{T_k} \le b_{T_k-1}$.

    Next, since $b_{T_k} \le b_{T_k-1}$, we have
    \begin{equation}\label{abrelation}
        a_{T_k} = T_kb_{T_k} - (T_k -1) b_{T_k-1} = T_k(b_{T_k}- b_{T_k-1})+ b_{T_k-1} \le b_{T_k-1} \le  \tau_{T_k -1}.
    \end{equation}
\end{proof}
\begin{remark}\label{construct_sub_sq}
    From the proof of Lemma~\ref{choose_sub_seq}, we may generate the index set $\{T_k\}$ by checking when $b_T \le b_{T-1}$ holds, where $b_T := \frac{1}{T} \sum_{k=1}^{T} a_k$. In addition, one can see from \eqref{abrelation} that the conclusion of Lemma~\ref{choose_sub_seq} actually holds for any $\{T_k\}$ with $T_k > 1$ that corresponds to a nonincreasing subsequence $\{b_{T_k}\}$ of $\{b_T\}$, i.e., $b_{T_k} \le b_{T_{k-1}}$ for all $k$. 
\end{remark}

\section{Algorithmic framework} \label{alg}

We present our algorithm for \eqref{problem} under the following general assumption and discuss its well-definedness.
We will impose further conditions in Section~\ref{convergence_section} to study further properties of the algorithm such as iteration complexity and subsequential convergence.
\begin{assumption}\label{assumption 1}
Consider \eqref{problem}.
    \begin{enumerate}[{\rm (i)}]
        \item $f: \mathbb{R}^n \rightarrow \mathbb{R}$ is Lipschitz differentiable. In other words, there exists $L>0$ such that
            \[
                \| \nabla f(x) - \nabla f(z) \| \le L\|x - z\|~~~\forall x,z\in \mathbb{R}^n.
            \]
        \item  $c: \mathbb{R}^n \rightarrow \mathbb{R}^m$ is continuously differentiable on $\mathbb{R}^n$ and there exists $L_c > 0$ such that
            \[
                \|J_c(x) - J_c(z)\| \le L_c \|x - z\|~~~\forall x, z\in \cl\left({\rm conv}(\dom g)\right),
            \]
            where $J_c$ denotes the Jacobian of $c$ and ${\rm conv}(\dom g)$ denotes the convex hull of $\dom g$.
            In addition, $J_c$ is bounded, i.e., there exists $M_c >0$ such that
            \[
                \|J_c(x) \| \le M_c~~~\forall x \in \cl\left( {\rm conv}(\dom g)\right).\footnote{We recall that this condition implies that $\|c(x) - c(y)\|\le M_c\|x-y\|$ for all $x \in \cl\left( {\rm conv}(\dom g)\right)$.}
            \]
        \item $g$ and $h$ are proper closed functions. Additionally, $g$ and $h$ are continuous on their domains,  $\dom g$ and $\dom h$ are closed, and it holds that
            \[
                \dom g \cap c^{-1}(\dom h) \neq \emptyset.
            \]
        Moreover, the proximal mappings of $\gamma g$ and $\gamma h$ are easy to compute for every $\gamma >0$.
        Finally, $\inf\{f+g\} > -\infty$ and $\inf h \ge 0$.
    \end{enumerate}
\end{assumption} 

In the case when $c$ is a {\em linear} map, the successive difference-of-convex approximation method (SDCAM) was proposed in \cite{liu2019successive} for solving \eqref{problem} that satisfies Assumption~\ref{assumption 1}. In each iteration of the SDCAM, one approximates the objective function by replacing the nonsmooth function $h$ by its Moreau envelope $e_{1/ \beta_t}h$ for some $\beta_t > 0$, i.e., 
\begin{equation}\label{hbeta}
    f(x) + g(x) + e_{1/ \beta_t}h(c(x)).
\end{equation}
The resulting subproblems are difference-of-convex (DC) optimization problems, thanks to the fact that Moreau envelopes are DC functions. These subproblems are then solved approximately to generate $x^t$ via the nonmonotone proximal gradient method with majorization (NPG$_{\rm major}$); see \cite[Appendix~A]{liu2019successive}. 
Assuming that $\{\beta_t\}$ is increasing with $\lim_{t\to\infty}\beta_t = \infty$ and that the subproblems are solved sufficiently accurately, it was shown in \cite{liu2019successive} that any accumulation point of $\{x^t\}$ is a stationary point of \eqref{problem}, under a standard constraint qualification. 

Here, we extend the SDCAM to handle \eqref{problem} under Assumption~\ref{assumption 1}. In particular, the $c$ is not necessarily linear as in \cite{liu2019successive}. In addition, to simplify the algorithmic design, we allow the flexibility for updating $\beta_t$ (and hence evolving the subproblem, which depends on the current iterate and $\beta_t$) {\em every iteration}, following recent works such as \cite{bohm2021variable,boct2020variable,deng2025single,XuPongSze25}. Our algorithm, which we call the single-loop successive difference-of-convex approximation method (SDCAM$_{\mathds{1}\ell}$), is presented in Algorithm~\ref{alg1} below. In this algorithm, in the $t$-th iteration, instead of minimizing \eqref{hbeta} approximately via the NPG$_{\rm major}$ as in \cite{liu2019successive}, we apply {\em one step} of a natural variant of the NPG$_{\rm major}$ to the following function:
\begin{equation}\label{hbeta2}
    f(x)+ g(x)+ (\beta_t/\beta_{t-1})e_{1/ \beta_{t-1}}h(c(x)).
\end{equation}
Specifically, we take advantage of the fact that the Moreau envelope $e_\lambda h$ with $\lambda > 0$ can be written as $e_\lambda h(y) = \frac1{2\lambda}\|y\|^2 - D_\lambda(y)$ for some real-valued convex function $D_\lambda$ (see \cite[Eq.~(6)]{liu2019successive}) and hence
\begin{equation}\label{subgradient22}
J_c(x)^T {\rm Prox}_{\lambda h}(c(x))\subseteq \lambda J_c(x)^T\partial D_\lambda(c(x))  = \lambda \partial (D_\lambda\circ c)(x),
\end{equation}
where the inclusion follows from the display before \cite[Eq.~(7)]{liu2019successive} (upon setting ${\cal A}$ to be the identity map there), and the equality follows from Proposition~8.12 (see also Corollary~8.11) and Theorem~10.6 of \cite{rockafellar2009variational}. The $x$-subproblem \eqref{get_xk_line} (for $t\ge 1$) is derived by rewriting the function in \eqref{hbeta2} as
\[
f(x) + ({\beta_t}/{2})\|c(x)\|^2 + g(x)- ({\beta_t}/{\beta_{t-1}})D_{1/\beta_{t-1}}(c(x)),
\]
linearizing, at $x^t$, the smooth part $f(\cdot) + \frac{\beta_t}{2}\|c(\cdot)\|^2$ using its gradient and $D_{1/\beta_{t-1}}(c(\cdot))$ using a subgradient obtained in \eqref{subgradient22}, and adding the proximal term $\|x-x^t\|^2/\mu$, where $\mu$ is found via backtracking to satisfy Condition~\ref{backtrack condition} (see the {\bf if}-loop). We then update $\{\beta_t\}$ judiciously. Here, we choose a positive nondecreasing $\{\beta_t\}$ such that $\lim_{t\to\infty}\beta_t = \infty$ and $\lim_{t\to\infty}\beta_t/\beta_{t-1}=1$; note that these conditions guarantee that $\frac{\beta_t}{\beta_{t-1}} e_{1/ \beta_{t-1}}h$ epi-converges to $h$ (see \cite[Proposition~7.4(d)]{rockafellar2009variational} and the discussion after its proof), which relates the set of minimizers of the subproblem objective \eqref{hbeta2} to that of the original objective in \eqref{problem} (see \cite[Section~7E]{rockafellar2009variational}). The use of \eqref{hbeta2} (instead of \eqref{hbeta}) plays a key role in our algorithmic design as it leads to the (pseudo-)descent lemma (see Lemma~\ref{lemma_H} below), which is the basis of our complexity and subsequential convergence analysis in Section~\ref{convergence_section}.

\begin{algorithm}[h]
	\caption{SDCAM$_{\mathds{1}\ell}$}
	\label{alg1}
	\begin{algorithmic}
		\STATE \textbf{Input:} Choose  $\mu_{\rm max} > \mu_{-1}>0$, $\rho \in (0,1)$, $\eta\ge 1$,  and  $x^0 \in \dom g$. Pick a positive nondecreasing sequence $\{\beta_t\}$ that satisfies $\lim_{t\to\infty}\beta_t = \infty$ and $\lim_{t\to\infty}\beta_t/\beta_{t-1}=1$.
        \\
        \STATE \textbf{Initialization:} $t = 0$, $\mu = \mu_{-1}$, $y^0 \in \dom h$.
		\STATE \textbf{Repeat} 
            \begin{align}
                 &\tilde{x} \in  \mathop{\rm Argmin}_{x\in \mathbb{R}^n}\left\{
                 \langle\nabla f(x^t)+ \beta_tJ_c(x^t)^T(c(x^t) -y^t),x \rangle + \frac1\mu\|x - x^t\|^2 + g(x)\right\}. \label{get_xk_line}\\
                &\textbf{if} \mbox{ Condition~\ref{backtrack condition} holds:} \nonumber \\
                &\ \ \ \ x^{t+1} = \tilde{x},\  \ \mu_t = \mu, \nonumber\\
                & \ \ \ \ y^{t+1} \in {\rm Prox}_{(1/ \beta_t) h}(c(x^{t+1})), \label{get_uk_line}\\
                &\ \ \ \  t \leftarrow t+1,\ \ \mu \leftarrow \min\{\mu_{\max},\eta\mu\} \quad (\textbf{successful iteration}) \nonumber\\
                &\textbf{else} \nonumber \\
                &\ \ \ \ \mu \leftarrow \rho \mu \quad (\textbf{unsuccessful iteration}) \nonumber
            \end{align}
        \STATE \textbf{Until }convergence
	\end{algorithmic}
\end{algorithm}

\begin{condition} \label{backtrack condition}
    \begin{enumerate}[{\rm (i)}]
            \item $\displaystyle \|c(\tilde{x}) - c(x^t)\| \le \sqrt{\frac{1}{\mu\beta_t}}\|\tilde{x} -x^t\|$;
        \item $\displaystyle f(\tilde{x}) + g(\tilde{x}) + \frac{\beta_{t}}{2}\|c (\tilde{x}) - y^t\|^2  \le f(x^t) + g(x^t) + \frac{\beta_{t}}{2}\| c(x^t) - y^t\|^2 - \frac{1}{2\mu}\| \tilde{x} - x^t\|^2 $.
    \end{enumerate}
\end{condition}

\subsection{Well-definedness and bounds on the number of unsuccessful iterations}

We will show that Algorithm~\ref{alg1} is well-defined in the sense that the number of unsuccessful iterations for each $t$ is finite. We will also derive a bound on the total number of unsuccessful iterations prior to obtaining $\mu_t$, for each $t\ge 0$. We start with the following auxiliary lemma.
\begin{lemma}\label{f+beta_quad}
    Suppose that Assumption~\ref{assumption 1} holds. Let $x$, $z\in \dom g$ and $y \in \mathbb{R}^m$. Then, for any $\beta>0$, it holds that 
    \begin{align}
        &\left(f(z)+ \frac{\beta}{2}\| c(z) -y\|^2\right) - 
        \left(f(x)+\frac{\beta}{2}\|c(x) -y\|^2\right) \nonumber \\
        &\le \left\langle \nabla f(x)+ \beta J_c(x)^T(c(x) - y), z-x \right\rangle + \frac{L+ ( L_c\|c(x) - y\|+ M_c^2 )\beta}{2}\|z - x\|^2. \nonumber
    \end{align}
\end{lemma}
\begin{proof}
    Notice that 
    \begin{align*}
        &\frac{\beta}{2}\|c(z) -y\|^2 =\frac{\beta}{2}\|c(z) -c(x)\|^2 + \beta\left\langle c(z) - c(x), c(x) -y \right\rangle + \frac{\beta}{2}\|c(x) - y\|^2 \\
        &\overset{(a)}{\le} \frac{M_c^2\beta}{2}\|z - x\|^2+  \beta\left\langle c(z) - c(x), c(x) -y \right\rangle + \frac{\beta}{2}\|c(x) - y\|^2 \\
        &= \frac{M_c^2\beta}{2}\|z - x\|^2+  \beta\left\langle c(z) - c(x)- J_c(x)(z-x), c(x) -y \right\rangle\\
        & \ \ \ \ +  \beta\left\langle J_c(x)(z-x), c(x) -y \right\rangle + \frac{\beta}{2}\|c(x) - y\|^2 \\
        & \le \frac{M_c^2\beta}{2}\|z - x\|^2+  \beta\|c(z) - c(x)- J_c(x)(z-x)\| \|c(x) -y\|\\
        & \ \ \ \ +  \beta\left\langle J_c(x)(z-x), c(x) -y \right\rangle + \frac{\beta}{2}\|c(x) - y\|^2 \\
        &\le \frac{\beta}{2}\|c(x) -y\|^2 + \left\langle \beta J_c(x)^T(c(x) -y), z-x \right\rangle + \frac{( L_c\|c(x) -y\|+M_c^2 )\beta}{2}\|z-x\|^2,
    \end{align*}
    where $(a)$ holds because of the definition of $M_c$ in Assumption~\ref{assumption 1}(ii), the last inequality holds because of the definition of $L_c$ in Assumption~\ref{assumption 1}(ii) and \cite[Lemma~3.2]{drusvyatskiy2019efficiency}.
    Next, since $f$ is Lipschitz differentiable with modulus $L$ (see Assumption~\ref{assumption 1}(i)), we have 
    \[
        f(z) \le f(x) +\langle \nabla f(x), z -x\rangle + \frac{L}{2}\|z -x\|^2. 
    \]
    The desired inequality now follows upon summing the above two displays.
\end{proof}
The next proposition establishes the well-definedness of Algorithm~\ref{alg1}.
\begin{proposition}[Well-definedness of Algorithm~\ref{alg1}] \label{bounded_mu}
    Suppose that Assumption~\ref{assumption 1} holds. Suppose that $x^t$, $y^t$ and $\mu_{t-1}$ are generated by Algorithm~\ref{alg1} for some $t\ge 0$. Then there are at most 
    \[
        \left \lceil -\log_{\rho}((L+(L_c\|c(x^t) - y^t\|+M_c^2)\beta_{t}  )\,\mu_{t-1}) -\log_{\rho} \eta \right \rceil
    \] 
     unsuccessful iterations before the next successful iteration. Moreover, it holds that
     \[
     \mu_t \ge \frac{\rho}{L+ (L_c\|c(x^t) - y^t\|+M_c^2)\beta_{t}}.
     \]
\end{proposition}
\begin{proof}
    Let 
    \[
        \bar n_t := \left \lceil -\log_{\rho} ((L+(L_c\|c(x^t) - y^t\|+M_c^2)\beta_{t} )\,\mu_{t-1}) -\log_{\rho} \eta\right \rceil.
    \]
    If $\bar n_t$ unsuccessful iterations are invoked, then, in view of the definition of $\mu_{t-1}$, we see that the $\mu$ immediately before the next successful iteration has to satisfy
    \[
        \mu \le \rho^{\bar n_t} \eta\mu_{t-1} \le \frac{1}{L+ (L_c\|c(x^t) - y^t\|+M_c^2)\beta_{t}}.
    \] 
    To complete the proof, it suffices to show that if $\mu \le \frac{1}{L+ (L_c\|c(x^t) - y^t\|+M_c^2)\beta_t }$, then this $\mu$ together with the corresponding $\tilde{x}$ in \eqref{get_xk_line} satisfies Condition~\ref{backtrack condition}. 
    To this end, notice that for this $\mu$ and the corresponding $\tilde{x}$ in \eqref{get_xk_line}, we have
    \begin{align}
        &f(\tilde{x})+ g(\tilde{x}) + \frac{\beta_{t}}{2}\| c(\tilde{x}) - y^t\|^2 \nonumber \\
        &\le f(x^t) + \langle \nabla f(x^t) + \beta_{t} J_c(x^t)^T(c(x^t) - y^t), \tilde{x} - x^t\rangle + g(\tilde{x}) + \frac{\beta_{t}}{2}\|c(x^t) - y^t\|^2 \nonumber  \\
        &\ \ \ \ + \frac{L+\left(L_c\| c(x^t) - y^t\|+M_c^2 \right)\beta_{t}}{2}\|\tilde{x} - x^t\|^2 \nonumber \\
        &\le f(x^t) + \langle \nabla f(x^t) + \beta_{t} J_c(x^t)^T(c(x^t) - y^t), \tilde{x} - x^t\rangle + \frac{1}{\mu}\|\tilde{x} - x^t\|^2+ g(\tilde{x})   \nonumber  \\
        & \ \ \ \  +\frac{\beta_{t}}{2}\|c(x^t) -y^t\|^2- \frac{1}{2\mu} \| \tilde{x} -x^ t\|^2  \nonumber \\
        & \le  f(x^t) + g(x^t)+ \frac{\beta_{t}}{2}\|c(x^t) - y^t\|^2 - \frac{1}{2\mu} \|\tilde{x} - x^t\|^2, \nonumber
    \end{align}
    where the first inequality holds upon letting $(z,x,y,\beta) = (\tilde{x},x^t,y^t,\beta_t)$ in Lemma~\ref{f+beta_quad}, the second inequality holds because $\mu \le \frac{1}{L+ (L_c\|c(x^t) - y^t\|+M_c^2)\beta_{t}}$ and the last inequality holds because of \eqref{get_xk_line}. Next, we also have
    \begin{align}
        \| c(\tilde{x}) - c(x^t)\| \le M_c \| \tilde{x} - x^t\| \le \sqrt{\frac{1}{\beta_t \mu}} \| \tilde{x} - x^t\|,   \nonumber
    \end{align}
    where the first inequality holds because of Assumption~\ref{assumption 1}(ii) and the last inequality holds because $\mu \le \frac{1}{L+ (L_c\|c(x^t) - y^t\|+ M_c^2)\beta_t}$.
\end{proof}
We next derive an important (pseudo-)descent lemma for the sequence generated by Algorithm~\ref{alg1}. This lemma will be invoked repeatedly in our analysis.
\begin{lemma} \label{lemma_H}
    Suppose that Assumption~\ref{assumption 1} holds. Define $H(x,\beta,y) := f(x) + g(x) + \frac{\beta}2\|c(x) - y\|^2 + h(y)$. Let $\{x^t\}$ and $\{y^t\}$ be generated by Algorithm~\ref{alg1}. Then, it holds that for all $t\ge 1$ and all $y\in \dom h$,
    \begin{align*}
        H(x^{t+1},\beta_t,y^t) 
        &\le H(x^t,\beta_{t-1},y) -\frac{1}{2\mu_t}\|x^{t+1} -x^t\|^2 + \frac{\beta_{t} -\beta_{t-1}}{2}\|c(x^t)-y^t\|^2. 
    \end{align*}
\end{lemma}
\begin{proof}
    Notice that for all $t\ge 1$, 
    \begin{align}
        &f(x^{t+1})+ g(x^{t+1})+ \frac{\beta_{t}}{2}\| c(x^{t+1}) -y^t\|^2 + h(y^t) \nonumber \\
        &\le f(x^t)+ g(x^t)+ \frac{\beta_{t}}{2}\|c(x^t) -y^t\|^2+h(y^t) -\frac{1}{2\mu_t}\|x^{t+1} -x^t\|^2 \nonumber \\
        &= f(x^t)+ g(x^t)+ \frac{\beta_{t-1}}{2}\|c(x^t) -y^t\|^2 +h(y^t) -\frac{1}{2\mu_t}\|x^{t+1} -x^t\|^2 + \frac{\beta_{t} -\beta_{t-1}}{2}\|c(x^t)-y^t\|^2 \nonumber\\
        &\le f(x^t)+ g(x^t) + \frac{\beta_{t-1}}{2}\|c(x^t) -y\|^2+ h(y) -\frac{1}{2\mu_t}\|x^{t+1} -x^t\|^2 + \frac{\beta_{t} -\beta_{t-1}}{2}\|c(x^t) -y^t\|^2, \nonumber
    \end{align}
    where the first inequality holds because $(x^{t+1},\mu^t)$ satisfies Condition~\ref{backtrack condition}(ii) in place of $(\tilde{x},\mu)$, the second inequality holds because of \eqref{get_uk_line}. This completes the proof.
\end{proof}

The next lemma bounds the deviations of $\{c(x^t)\}$ from $\{y^{t-1}\}$ and $\{y^t\}$, respectively.
\begin{lemma} 
    Suppose that Assumption~\ref{assumption 1} holds. Let $\{x^t\}$ and $\{y^t\}$ be generated by Algorithm~\ref{alg1}. Then, the following statements hold.
    \begin{enumerate}[{\rm (i)}]
        \item For all $t \ge 1$,  it holds that 
            \begin{align}
               \!\!\! \frac{1}{2}\|c(x^t) - y^{t-1}\|^2 \!\le\! \frac{1}{\beta_0}( f(x^1)+ g(x^1) - \inf\{f + g\}) \!+\! \frac{1}{2}\|c(x^1) -y^0\|^2 \!+\! \frac{1}{\beta_0}h(y^0).     \label{bd_xt-yt-1}
            \end{align}
        \item For all $t \ge 1$,  it holds that 
            \begin{align}
                \frac{1}{2}\|c(x^t) - y^t\|^2 \le \frac{2}{\beta_0}( f(x^1)+ g(x^1) - \inf\{f + g\}) + \|c(x^1) -y^0\|^2 + \frac{2}{\beta_0}h(y^0).     \label{bd_xt-yt}
            \end{align}
    \end{enumerate}
\end{lemma}
\begin{proof}
   Applying Lemma~\ref{lemma_H} with $y=y^t$, we see that for all $t\ge 1$,
   \[
   H(x^{t+1},\beta_t,y^t) 
        \le H(x^t,\beta_{t-1},y^t) -\frac{1}{2\mu_t}\|x^{t+1} -x^t\|^2 + \frac{\beta_{t} -\beta_{t-1}}{2}\|c(x^t)-y^t\|^2. 
   \]
   Subtracting $\inf\{f+g\}$ from both sides of the above inequality and then dividing both sides by $\beta_t$, we have, upon invoking the definition of $H$ that for all $t\ge 1$,
    \begin{align*}
        &\frac{1}{\beta_t} ( f(x^{t+1})+ g(x^{t+1})- \inf\{f + g\}) + \frac{1}{2}\|c(x^{t+1}) - y^t\|^2 + \frac{1}{\beta_t}h(y^t) \nonumber \\
        &\le \frac{1}{\beta_t} ( f(x^t)+ g(x^t)-\inf\{f + g\}) + \frac{1}{2}\| c(x^t) -y^{t}\|^2 +\frac{1}{\beta_t}h(y^{t}) - \frac{1}{ 2\mu_t\beta_t}\|x^{t+1} -x^t\|^2 \nonumber \\
        & \le \frac{1}{\beta_{t-1}} ( f(x^t)+ g(x^t)-\inf\{f + g\}) + \frac{1}{2}\| c(x^t) -y^{t}\|^2 +\frac{1}{\beta_{t-1}}h(y^{t}) - \frac{1}{ 2\mu_t\beta_t}\|x^{t+1} -x^t\|^2 \nonumber \\
        & \le \frac{1}{\beta_{t-1}} ( f(x^t)+ g(x^t)-\inf\{f + g\}) \!+\! \frac{1}{2}\| c(x^t) -y^{t-1}\|^2 \!+\!\frac{1}{\beta_{t-1}}h(y^{t-1}) \!-\! \frac{1}{ 2\mu_t\beta_t}\|x^{t+1} -x^t\|^2,
    \end{align*}
    where the second inequality holds because $h$ is nonnegative (see Assumption~\ref{assumption 1}(iii)), $f(x^t) + g(x^t) \ge \inf\{f + g\}$ for all $t\ge 0$ and $\{\beta_t\}$ is nondecreasing, and the last inequality holds because of \eqref{get_uk_line}.
    Define $\Theta (x, \beta, y) := \frac{1}{\beta} \left(f(x)+ g(x) -\inf\{f + g\} \right)+ \frac{1}{2}\|c(x) -y\|^2 + \frac{1}{\beta}h(y)$. Then we see from the above display that
$\Theta(x^{t+1}, \beta_t, y^t) \le \Theta(x^t, \beta_{t-1}, y^{t-1})$ for all $t\ge 1$.
    Therefore, we have 
    \begin{equation}\label{hehehaha}
        \Theta(x^{t+1}, \beta_{t}, y^t) \le \Theta(x^1, \beta_0, y^0) \ \ \forall t\ge 0.
    \end{equation}
    Combining this with the observation that $\frac{1}{2}\|c(x^{t+1}) -y^t\|^2 \le \Theta(x^{t+1}, \beta_{t}, y^t)$ for all $t\ge 0$ proves \eqref{bd_xt-yt-1}.
   
   Next, we have from the definition of $\Theta$ and \eqref{hehehaha} that
   \begin{align}
       &\frac{1}{\beta_t}h(y^t)  \le \Theta(x^{t+1}, \beta_t, y^t) \le \Theta(x^1, \beta_0, y^0)\ \ \ \ \forall t\ge 0.  \label{bd_hyt} 
   \end{align}
Then we deduce further that for all $t\ge 1$,
   \begin{align}
       \frac{1}{2}\|c(x^t) - y^t\|^2  
       &\overset{(a)}{\le} \frac{1}{2}\|c(x^t) -y^t\|^2 + \frac{1}{\beta_{t-1}}h(y^t)  \overset{(b)}{\le} \frac{1}{2}\|c(x^t) -y^{t-1}\|^2 + \frac{1}{\beta_{t-1}}h(y^{t-1}) \nonumber \\
       &\le  \frac{2}{\beta_0}\left( f(x^1)+ g(x^1)-  \inf\{f + g\}\right) + \|c(x^1) -y^0\|^2 + \frac{2}{\beta_0}h(y^0), \nonumber
   \end{align}
   where $(a)$ holds because $h$ is nonnegative, $(b)$ holds because of \eqref{get_uk_line}, and the last inequality follows from \eqref{bd_xt-yt-1} and \eqref{bd_hyt}. This proves \eqref{bd_xt-yt}.
\end{proof}
We are now ready to derive a bound on the total number of unsuccessful iterations prior to the $t$-th successful iteration of Algorithm~\ref{alg1}, for each $t\ge 0$.
\begin{proposition}[Number of unsuccessful iterations]
    Suppose that Assumption~\ref{assumption 1} holds. For each $t\ge 0$, let $U(t)$ be the total number of unsuccessful iterations prior to the $t$-th successful iteration of Algorithm~\ref{alg1}. 
    Then
    \[
        U(t) = t\log_{1/\rho} \eta+ \mathcal{O}\left( \ln \beta_t\right) .
    \]
    In addition, if we define  
    \begin{align}\label{M0}  
    \!\!\!\!M_0 \!:=\! \max\left\{\!\|c(x^0)\!-\!y^0\|,\!\sqrt{\frac{4}{\beta_0}( f(x^1)\!+\! g(x^1) \!-\! \inf\{f + g\}) \!+\! 2\|c(x^1) \!-\!y^0\|^2 \!+\! \frac{4}{\beta_0}h(y^0)}\right\}\!,\!\!\!\!
    \end{align}
    then 
    \begin{align}
       {\mu_t}^{-1} \le \rho^{-1}L+ \rho^{-1}\left( L_cM_0 +M_c^2\right)\beta_t. \label{bd_1/mu_t}
    \end{align}
\end{proposition}
\begin{proof} 
    Let $n_t$ be the number of unsuccessful iterations between the $(t-1)$-th and the $t$-th successful iteration. Note that $n_t$ is well-defined and finite thanks to Proposition~\ref{bounded_mu}. Now, in view of the definition of $\mu_{t}$, we see that
    \[
    \begin{cases}
    &\mu_0 = \rho^{n_0}\mu_{-1},\\
    &\mu_t = \rho^{n_t} \min\{\mu_{\max},\eta\mu_{t-1}\} \le \rho^{n_t}\eta \mu_{t-1} \le \eta^t \rho^{\sum_{k=1}^t n_k} \mu_0 = \eta^t \rho^{\sum_{k=0}^t n_k} \mu_{-1}\ \ \ \forall t\ge 1.
    \end{cases}        
    \]
    Consequently, we have $\mu_t \le \eta^t \rho^{\sum_{k=0}^t n_k} \mu_{-1} = \eta^t \rho^{U(t)} \mu_{-1}$ for all $t\ge 0$.
    On the other hand, by Proposition~\ref{bounded_mu}, \eqref{bd_xt-yt} and the definition of $M_0$ in \eqref{M0}, we see that
    \[
        \mu_t \ge \frac{\rho}{L+(L_cM_0+ M_c^2)\beta_t}\ \ \ \forall t\ge 0.
    \]
    Hence, we have $\eta^t \rho^{U(t)} \mu_{-1}\ge \frac{\rho}{L+(L_cM_0+ M_c^2)\beta_t}$, which gives
    \[
        U(t) \le 1+ \lceil \log_{1/ \rho} \left( L+ (L_cM_0+ M_c^2)\beta_t \right) + \log_{1/ \rho}\mu_{-1} + t \log_{1/ \rho} \eta \rceil.
    \]
    This completes the proof.
\end{proof}

Before ending this subsection, we present the optimality conditions for the subproblems that arise in Algorithm~\ref{alg1}, which will be used repeatedly in our analysis below. Suppose Assumption~\ref{assumption 1} holds and let $\{x^t\}$, $\{y^t\}$ and $\{\mu_t\}$ be generated by Algorithm~\ref{alg1} (these sequences are well defined in view of Proposition~\ref{bounded_mu}). Then we have from \eqref{get_xk_line}, \eqref{get_uk_line} and \cite[Exercise~8.8(c)]{rockafellar2009variational} that for all $t\ge 1$,
 \begin{align}\label{yh_optcon}
 \begin{cases}
        & 0 \in \nabla f(x^t) + \beta_tJ_c(x^t)^T(c(x^t) -y^t) + \frac{2}{\mu_t} (x^{t+1} -x^t) + \partial g(x^{t+1}),\\
        & 0 \in \beta_{t-1} (y^{t} - c(x^t))+ \partial h(y^t). 
\end{cases}
    \end{align}

\subsection{Subsequential convergence under a standard constraint qualification}

The following theorem analyzes convergence of a certain subsequence of the sequence generated by Algorithm~\ref{alg1} under the standard constraint qualification \eqref{CQ}. It imposes seemingly restrictive assumptions including {\em vanishing scaled successive changes} along some subsequence and the existence of an accumulation point satisfying $c(x^*) = y^*$: we will argue in Section~\ref{convergence_section} that these additional conditions can be satisfied by suitably choosing $\{\beta_t\}$ in Algorithm~\ref{alg1}, under additional structural assumptions on $g$ and $h$. As we will see in Section~\ref{convergence_section}, this theorem is useful for {\em translating} global complexity bounds into asymptotic convergence results along some subsequences.
\begin{theorem}\label{convergence}
    Suppose that Assumption~\ref{assumption 1} holds, and $\{x^t\}$, $\{y^t\}$ and $\{\mu_t\}$ are generated by Algorithm~\ref{alg1}. Suppose that $\{x^{T_k}\}$ is a subsequence of $\{x^t\}$ satisfying $\frac{1}{\mu_{T_k}}\|x^{T_k+1} -x^{T_k}\| \rightarrow 0$. Let $(x^*, y^*)$ be an accumulation point of $\{(x^{T_k}, y^{T_k})\}$ that satisfies $c(x^*) = y^*$. If the following condition holds
    \begin{equation}\label{CQ}
        \begin{array}{cc}
             & \zeta \in \partial^{\infty} g(x^*),~\eta\in \partial^{\infty} h(c(x^*)),~\zeta + J_c(x^*)^T\eta =0   \\ [3 pt]
             &  \Longrightarrow \zeta =0,~\eta=0,
        \end{array}
    \end{equation}
    then
        \begin{align*}
            0 \in \nabla f(x^*) + \partial g(x^*) + J_c(x^*)^T\partial h(c(x^*)).
        \end{align*}
\end{theorem}
\begin{proof}
    By assumption, there exists a subsequence $\{(x^{T_{k_i}}, y^{T_{k_i}})\}$ (we assume without loss of generality that $T_{k_i} \ge 1$) of $\{(x^{T_k}, y^{T_k})\}$ such that  $\lim_{i \rightarrow \infty} x^{T_{k_i}} = x^*$ and $\lim_{i\rightarrow \infty} y^{T_{k_i}} = y^*$. Moreover, since $\mu_k \le \mu_{\max} < \infty$ by construction, we also have 
    \begin{equation}\label{prefootnote}
    \|x^{T_{k_i}+1} - x^{T_{k_i}}\|\le \mu_{\max}\cdot\mu^{-1}_{T_{k_i}}\|x^{T_{k_i}+1} - x^{T_{k_i}}\| \to 0.
    \end{equation}
    
    Now,
    in view of \eqref{yh_optcon}, we see that there exists
    $\psi_{T_{k_i}+1} \in \partial g(x^{T_{k_i}+1})$ satisfying
    \begin{align}  \label{subpro_optimal}
        0 &=\nabla f(x^{T_{k_i}}) + \beta_{T_{k_i}}J_c(x^{T_{k_i}})^T(c(x^{T_{k_i}}) -y^{T_{k_i}})  +\psi_{T_{k_i}+1}  +\frac{2}{\mu_{T_{k_i}}} (x^{T_{k_i}+1} -x^{T_{k_i}});
    \end{align}
    and it holds that
    \begin{align}\label{subdiff_h_u}
        \beta_{T_{k_i}-1} (c(x^{T_{k_i}}) - y^{T_{k_i}}) \in \partial h(y^{T_{k_i}}).
    \end{align}
    
    Define
    \begin{align}
        r_{T_{k_i}} := \|\beta_{T_{k_i}}J_c(x^{T_{k_i}})^T(c(x^{T_{k_i}}) -y^{T_{k_i}})\| + \| \psi_{T_{k_i}+1} \|. \nonumber
    \end{align}    
    We claim that $\{r_{T_{k_i}} \}$ is bounded. 
    
    Suppose to the contrary that $\{r_{T_{k_i}}\}$ is unbounded. By passing to a further subsequence if necessary, we can assume without loss of generality that $\lim_{i \rightarrow \infty} r_{T_{k_i}} =\infty$ and $\inf_{i\ge 0}r_{T_{k_i }} >0$. Hence, the sequences $\{{\beta_{T_{k_i}}J_c(x^{T_{k_i}})^T(c(x^{T_{k_i}}) -y^{T_{k_i}})}/{r_{T_{k_i}}}\}$ and $\{{\psi_{T_{k_i}+ 1} }/{r_{T_{k_i}}}\}$ are bounded. Passing to a further subsequence if necessary, we have
    \begin{align} \label{limit}
        \lim_{i \rightarrow\infty} \frac{\beta_{T_{k_i}}J_c(x^{T_{k_i}})^T(c(x^{T_{k_i}}) -y^{T_{k_i}})}{r_{T_{k_i}}} = \xi^*\ \ {\rm and}\ \ \lim_{i\rightarrow\infty} \frac{\psi_{T_{k_i}+1}}{r_{T_{k_i}}} = \zeta^*,
    \end{align}
    for some $\xi^*$ and $\zeta^*$.
    In addition, by the definitions of $\{r_{T_{k_i}}\}$, $\xi^*$ and $\zeta^*$, we obtain that
    \begin{align}\label{xi+zeta_norm}
        1= \|\xi^*\| + \|\zeta^*\|.
    \end{align}
    Additionally, combining  \eqref{subpro_optimal} and \eqref{limit} with the facts that $\mu^{-1}_{T_{k_i}}\|x^{T_{k_i}+1} - x^{T_{k_i}}\| \rightarrow 0$ and $r_{T_{k_i}}\to \infty$, we have
    \begin{align} \label{xi+zeta}
        0 = \xi^* +\zeta^*.
    \end{align}
    Now, using the definition of $\zeta^*$, the facts that $r_{T_{k_i}}\to \infty$ and $\|x^{T_{k_i}+1} - x^{T_{k_i}}\| \rightarrow 0$ (see \eqref{prefootnote}), and the continuity of $g$ on its closed domain, we deduce that
    \begin{align} \label{zeta_horizon_sub}
        \zeta^* \in \partial^{\infty} g(x^*).
    \end{align}
    Next, we will prove that $\xi^* \in J_c(x^*)^T\partial^{\infty} h(c(x^*))$.
    To this end, we first claim that the sequence defined as
    \[
        \vartheta_{T_{k_i}} :=  \left\| \frac{\beta_{T_{k_i}}(c(x^{T_{k_i}}) - y^{T_{k_i}})}{r_{T_{k_i}}} \right\|
    \]
    is bounded.
    Suppose to the contrary that it is unbounded. By passing to a further subsequence if necessary, we may assume that $\vartheta_{T_{k_i}}\to \infty$, $\inf_i \vartheta_{T_{k_i}} > 0$ and there exists $\eta^*$ such that
    \begin{align}\label{lim_theta_tk}
        \lim_{i \rightarrow \infty} \frac{1}{\vartheta_{T_{k_i}}} \frac{\beta_{T_{k_i}}(c(x^{T_{k_i}}) -y^{T_{k_i}})}{r_{T_{k_i}}} = \eta^*.
    \end{align}
    Then due to the definition of $\eta^*$ and \eqref{limit}, we obtain that
    \begin{align} \label{accumu_point_eta}
        \| \eta^* \| = 1~\text{and} ~J_c(x^*)^T \eta^* = 0.
    \end{align}
    In addition, from \eqref{lim_theta_tk}, we can also deduce that
    \begin{align*}
        \eta^* &= \lim_{i \rightarrow \infty} \frac{1}{\vartheta_{T_{k_i}}} \frac{\beta_{T_{k_i}}(c(x^{T_{k_i}})- y^{T_{k_i}})}{r_{T_{k_i}}} 
        = \lim_{i \rightarrow \infty} \frac{1}{\vartheta_{T_{k_i}}r_{T_{k_i}}}\frac{\beta_{T_{k_i}}}{\beta_{T_{k_i}-1}} \beta_{T_{k_i}-1}\left(c(x^{T_{k_i}}) - y^{T_{k_i}}\right) \nonumber \\
        &\overset{(a)}\in \left\{ \lim_{i \rightarrow \infty} \frac{1}{\vartheta_{T_{k_i}}r_{T_{k_i}}} \frac{\beta_{T_{k_i}}}{\beta_{T_{k_i}-1}} w^{i}: w^{i} \in \partial h(y^{T_{k_i}}) \mbox{ for each }i\right\} \subseteq \partial^{\infty} h(c(x^*)),
    \end{align*}
    where $(a)$ follows from \eqref{subdiff_h_u}, and the last inclusion holds because of the definition of horizon subdifferential, the continuity of $h$ on its closed domain, and the facts that $\lim_{i \rightarrow \infty} y^{T_{k_i}}= y^* = c(x^*)$, $\frac{\beta_{T_{k_i}}}{\beta_{T_{k_i}-1}}\to 1$ and $\vartheta_{T_{k_i}}r_{T_{k_i}}\to \infty$.
    The above display together with \eqref{accumu_point_eta} contradicts \eqref{CQ}. Therefore, $\{{\beta_{T_{k_i}}(c(x^{T_{k_i}}) - y^{T_{k_i}})}/{r_{T_{k_i}}}\}$ is bounded. Now we can invoke \eqref{limit} to deduce that
    \begin{align*}
        &\xi^* \!=\! \lim_{i \rightarrow \infty}\! \frac{\beta_{T_{k_i}}J_c(x^{T_{k_i}})^T(c(x^{T_{k_i}}) - y^{T_{k_i}})}{r_{T_{k_i}}} \!=\! \lim_{i \rightarrow \infty}\! \frac{1}{r_{T_{k_i}}}\frac{\beta_{T_{k_i}}}{\beta_{T_{k_i}-1}} \beta_{T_{k_i}\!-1}J_c(x^{T_{k_i}})^T\!(c(x^{T_{k_i}}) - y^{T_{k_i}})\nonumber\\
        &\overset{(a)}{\in} J_c(x^*)^T\left\{ \lim_{i \rightarrow \infty} \frac{1}{r_{T_{k_i}}}\frac{\beta_{T_{k_i}}}{\beta_{T_{k_i}-1}}  w^{i}: w^{i} \in \partial h(y^{T_{k_i}}) \mbox{ for each }i \right\} \subseteq J_c(x^*)^T\partial^{\infty} h(c(x^*)),
    \end{align*}
    where $(a)$ follows from \eqref{subdiff_h_u} and the boundedness of $\{{\beta_{T_{k_i}}(c(x^{T_{k_i}}) - y^{T_{k_i}})}/{r_{T_{k_i}}}\}$, and the last inclusion follows from the definition of horizon subdifferential, the continuity of $h$ on its closed domain, and the facts that $\lim_{i \rightarrow \infty} y^{T_{k_i}}= y^* = c(x^*)$, $\frac{\beta_{T_{k_i}}}{\beta_{T_{k_i}-1}}\to 1$ and $r_{T_{k_i}}\to \infty$.
    The above display together with \eqref{xi+zeta_norm}, \eqref{xi+zeta} and \eqref{zeta_horizon_sub} contradicts \eqref{CQ}. Thus,  $\{r_{T_{k_i}}\}$ is bounded.

    Since $\{r_{T_{k_i}}\}$ is bounded, by passing to a further subsequence if necessary, we assume without loss of generality that
\begin{align} \label{limit22}
        \lim_{i \rightarrow\infty} \beta_{T_{k_i}}J_c(x^{T_{k_i}})^T(c(x^{T_{k_i}}) -y^{T_{k_i}}) = \tilde{\xi}^*\ \ {\rm and}\ \ \lim_{i\rightarrow\infty} \psi_{T_{k_i}+1} = \tilde{\zeta}^*,
    \end{align}
    for some $\tilde{\xi}^*$ and $\tilde{\zeta}^*$. Since $\frac{\beta_{T_{k_i}}}{\beta_{T_{k_i}-1}}\to 1$, we see from the first equality in \eqref{limit22} that
   \begin{align} \label{limit23}
        \lim_{i \rightarrow\infty} \beta_{T_{k_i}-1}J_c(x^{T_{k_i}})^T(c(x^{T_{k_i}}) -y^{T_{k_i}}) = \tilde{\xi}^*.
    \end{align}   
    
    We claim that  
    \begin{align}\label{tilde_xi}
        \tilde{\xi}^* \in J_c(x^*)^T\partial h(c(x^*))\ \ {\rm and }\ \ \tilde{\zeta}^* \in \partial g(x^*).
    \end{align}
    Note that the second inclusion in \eqref{tilde_xi} holds because of the fact that $\psi_{T_{k_i}+1} \in \partial g(x^{T_{k_i}+1})$, the second equality in \eqref{limit22}, the continuity of $g$ on its closed domain and  $\|x^{T_{k_i}+1} - x^{T_{k_i}}\| \rightarrow 0$  (see \eqref{prefootnote}).
    We now prove the first inclusion in \eqref{tilde_xi}. Define
    \[
    \iota_{T_{k_i}}:= \|\beta_{T_{k_i}-1}(c(x^{T_{k_i}}) -y^{T_{k_i}})\|.
    \]    
    We claim that $\{\iota_{T_{k_i}}\}$ is bounded. Suppose to the contrary that $\{\iota_{T_{k_i}}\}$ is unbounded. By passing to a further subsequence if necessary, we may assume that $\iota_{T_{k_i}}\to \infty$, $\inf_i \iota_{T_{k_i}} > 0$ and there exists $\tilde\eta^*$ such that
    \begin{align}\label{lim_iota_tk}
        \lim_{i \rightarrow \infty} \frac{1}{\iota_{T_{k_i}}} \beta_{T_{k_i}-1}(c(x^{T_{k_i}}) -y^{T_{k_i}}) = \tilde\eta^*.
    \end{align}
    In view of the definition of $\tilde\eta^*$ and \eqref{limit23}, we obtain that
    \begin{align} \label{accumu_point_iota}
        \| \tilde\eta^* \| = 1~\text{and} ~J_c(x^*)^T \tilde\eta^* = 0.
    \end{align}
    Moreover, from \eqref{lim_iota_tk}, we can also deduce that
    \begin{align*}
        \tilde\eta^* &= \lim_{i \rightarrow \infty} \frac{1}{\iota_{T_{k_i}}} \beta_{T_{k_i}-1}(c(x^{T_{k_i}}) -y^{T_{k_i}})
       \\
       &\overset{(a)}\in \left\{ \lim_{i \rightarrow \infty} \iota^{-1}_{T_{k_i}} w^i: w^{i} \in \partial h(y^{T_{k_i}}) \mbox{ for each }i\right\} \subseteq \partial^{\infty} h(c(x^*)),
    \end{align*}
    where $(a)$ follows from \eqref{subdiff_h_u}, and the last inclusion follows from the definition of horizon subdifferential, the continuity of $h$ on its closed domain, and the facts that $\lim_{i \rightarrow \infty} y^{T_{k_i}}= y^* = c(x^*)$ and $\iota_{T_{k_i}}\to \infty$. The above display together with \eqref{accumu_point_iota} contradicts \eqref{CQ}, showing that $\{\iota_{T_{k_i}}\}$ is bounded. 
 Now we can invoke \eqref{limit23} to deduce that
    \begin{align*}
        \tilde{\xi}^* &= \lim_{i \rightarrow \infty} \beta_{T_{k_i}-1}J_c(x^{T_{k_i}})^T(c(x^{T_{k_i}}) -y^{T_{k_i}}) \nonumber\\
        &\overset{(a)}{\in} J_c(x^*)^T\left\{ \lim_{i \rightarrow \infty} w^{i}: w^{i} \in \partial h(y^{T_{k_i}}) \mbox{ for each }i \right\} \subseteq J_c(x^*)^T\partial h(c(x^*)),
    \end{align*}
    where $(a)$ follows from \eqref{subdiff_h_u} and the boundedness of $\{\iota_{T_{k_i}}\}$, and the last inclusion follows from the definition of horizon subdifferential, the continuity of $h$ on its closed domain, and the fact that $\lim_{i \rightarrow \infty} y^{T_{k_i}}= y^* = c(x^*)$. This establishes the first inclusion in \eqref{tilde_xi}.
    
    Finally, we pass to the limit on both sides of \eqref{subpro_optimal} to obtain
    \begin{align*}
        0 &= \nabla f(x^*) + \lim_{i \rightarrow \infty} \Big[ \beta_{T_{k_i}}J_c(x^{T_{k_i}})^T(c(x^{T_{k_i}}) -y^{T_{k_i}})  +\psi_{T_{k_i}+1}  +\frac{2}{\mu_{T_{k_i}}} (x^{T_{k_i}+1} -x^{T_{k_i}}) \Big] \\
        &\overset{(a)}{=}\nabla f(x^*) + \lim_{i \rightarrow \infty} \left[\beta_{T_{k_i}}J_c(x^{T_{k_i}})^T(c(x^{T_{k_i}}) -y^{T_{k_i}}) +\psi_{T_{k_i}+1}\right] \\
        &\overset{(b)}{\in} \nabla f(x^*) + \partial g(x^*) + J_c(x^*)^T\partial h(c(x^*)),
    \end{align*}
    where $(a)$ holds as $\mu^{-1}_{T_{k_i}}\|x^{T_{k_i}+1} - x^{T_{k_i}}\| \rightarrow 0$, and $(b)$ follows from \eqref{limit22} and \eqref{tilde_xi}.
\end{proof}

\section{Global complexity and subsequential convergence} \label{convergence_section}

In this section, we analyze the iteration complexity and the global convergence properties of Algorithm~\ref{alg1} under additional assumptions on $g$, $h$ and $\{\beta_t\}$, on top of the basic Assumption~\ref{assumption 1}.

We start by introducing the following assumption concerning $\{\beta_t\}$.\footnote{Notice that a $\{\beta_t\}$ satisfying Assumption~\ref{assumption_beta} clearly also satisfies the implicit assumptions required by Algorithm~\ref{alg1} for $\{\beta_t\}$, i.e., $\beta_t$ is positive nondecreasing with $\beta_t\to\infty$ and $\beta_t/\beta_{t-1}\to 1$.}
\begin{assumption}\label{assumption_beta}
    There exist $\delta\in (0,1)$ and positive constants $\alpha_0$, $\gamma_0$ and $\eta_0$ with $0<\alpha_0 \leq \gamma_0$ such that $\{\beta_t\}$ satisfies:
    \begin{enumerate}[{\rm (i)}]
        \item $\alpha_0(t+1)^{\delta} \leq \beta_t \leq \gamma_0(t+1)^{\delta}$ for all $t \geq 0$,
        \item $0\le\beta_t - \beta_{t-1} \leq \eta_0 t^{\delta-1}$ for all $t \geq 1$.
    \end{enumerate}
\end{assumption}

The next proposition describes a specific sequence satisfying Assumption~\ref{assumption_beta}.
\begin{proposition}
    Let $\beta_0 >0$, $\delta \in (0,1)$ and fix an integer $K\ge 1$. Define, for $t \ge 1$,
        \[
            \beta_t := \begin{cases}
                \beta_0 (t+1)^{\delta} & \text{if } \operatorname{mod}(t, K) = 0, \\
                \beta_0 (nK + 1)^{\delta} & \text{if } nK < t < (n+1)K \text{ for some nonnegative integer } $n$.
            \end{cases}
        \]  
    Then the sequence $\{\beta_t\}$ satisfies Assumption~\ref{assumption_beta}.
\end{proposition}
\begin{proof}
    Observe that for any $t \geq 0$, if $t = nK$ for some integer $n$, then $\beta_t = \beta_0 (nK + 1)^\delta = \beta_0 (t + 1)^\delta$. If $nK < t < (n+1)K$, then $\beta_t = \beta_0 (nK + 1)^\delta$. Since $nK + 1 \leq t + 1$, in this case, we deduce that $\beta_t \leq \beta_0 (t + 1)^\delta$. Therefore, the relation $\beta_t \leq \beta_0 (t+1)^\delta$ holds for all $t$.    
    Now we derive a lower bound for $\beta_t$ when $nK < t< (n+1)K$. In this case,
    \begin{align*}
        \beta_t &= \beta_0 (nK + 1)^\delta \ge \beta_0 \left( \frac{nK + 1}{nK + K} \right)^\delta (t + 1)^\delta \ge \beta_0 \left( \frac{1}{K} \right)^\delta (t + 1)^\delta,
    \end{align*}
    where the first inequality holds because $t+1 \le nK+K $, the last inequality holds because $K\ge 1$ and the function $x \mapsto \left(\frac{x+1}{x+K} \right)^{\delta}$ is nondecreasing for $x\ge 0$.
    Thus, upon letting $\alpha_0 = \beta_0 K^{-\delta}$ and $\gamma_0 = \beta_0$, we see that $\alpha_0 \leq \gamma_0$ and Assumption~\ref{assumption_beta}(i) is satisfied.

    Next, notice that if $nK \le t-1<t< (n+1)K$ for some nonnegative integer $n$, then we have $\beta_t - \beta_{t-1} = 0$. Now, consider the case when $nK \le t-1< (n+1)K= t$ for some nonnegative integer $n$. In this case,
    \begin{align*}
        &\beta_t -\beta_{t-1} = \beta_0 \left( ((n+1)K + 1)^\delta - (nK + 1)^\delta \right)\le \beta_0 \delta (nK + 1)^{\delta-1} K \nonumber \\
        &= \beta_0\delta K\left( \frac{nK+1}{nK+K} \right)^{\delta-1}\left( nK+K \right)^{\delta-1} \le \beta_0\delta K\cdot K^{1-\delta}\left( nK+K \right)^{\delta-1} = \beta_0\delta  K^{2-\delta}t^{\delta-1},
    \end{align*}
    where the first inequality holds because $x\mapsto x^{\delta}$ is concave, and the last inequality holds because $K\ge 1$ and $x \mapsto \left(\frac{x+K}{x+1} \right)^{1-\delta}$ is nonincreasing for $x \ge 0$.
    Therefore, we see that Assumption~\ref{assumption_beta}(ii) holds with $\eta_0 = \beta_0 \delta K^{2-\delta}$. 
\end{proof}

We next discuss the iteration complexity and global convergence properties of Algorithm~\ref{alg1} in the next three subsections, where we progressively relax our assumptions on $h$.

\subsection{Convergence analysis when $h$ is Lipschitz continuous}\label{sec41}

In this subsection, we consider the following Lipschitz continuity assumption on $h$, which has been widely studied in the literature; see Table~\ref{table:smoothing_methods}.
\begin{assumption} \label{asm_lip_h}
    There exists a constant $M_h >0$ such that $h$ in \eqref{problem} satisfies
    \[
        |h(v) - h(y)| \le M_h\|v -y\|\ \ \ \ \ \forall v,y\in \mathbb{R}^m.
    \]
\end{assumption}
This assumption, together with Assumption~\ref{assumption 1}, is naturally satisfied in a number of practical applications. As a concrete example, consider the following model for multiple-input-multiple-output (MIMO) signal detection with $p$-ary phase-shift
keying (PSK) (see, e.g., \cite{kume2025proximal}):
\[
   \min\limits_{(r, \theta)\in [\underline{r}, 1]^n \times \mathbb{R}^n} \frac{1}{2}\| \hat{y}- A \varphi(r, \theta)\|^2 + \lambda_1 \sum_{i=1}^n \frac{1}{r_i} + \lambda_2\left\| {\bf sin}\left(\frac{p\theta}{2}\right) \right\|_1,
\]
where $\lambda_1>0$, $\lambda_2>0$, $\underline{r}\in(0,1]$, $\hat{y} \in \mathbb{R}^{2m}$, $A \in \mathbb{R}^{2m \times 2n}$, and $\varphi: \mathbb{R}^n \times \mathbb{R}^n \rightarrow \mathbb{R}^{2n}$ is defined as 
\[
    \varphi(r, \theta) := \left[ \begin{array}{c}
         r\odot {\bf cos}(\theta)  \\
         r\odot {\bf sin}(\theta) 
    \end{array} \right],
\]
where $\odot$ denotes the entry-wise product (Hadamard product), ${\bf sin}(\theta) := \begin{bmatrix} 
\sin \theta_1& \sin \theta_2& \cdots &\sin \theta_n
\end{bmatrix}^T$, 
${\bf cos}(\theta) := \begin{bmatrix}
\cos \theta_1& \cos \theta_2& \cdots& \cos \theta_n
\end{bmatrix}^T$.
This model is an instance of \eqref{problem} satisfying Assumptions~\ref{assumption 1} and \ref{asm_lip_h}. Specifically, one can take $f(r, \theta) := \frac{1}{2}\|\hat{y}- A \varphi(r, \theta)\|^2+ \lambda_1 \sum_{i=1}^n \gamma(r_i) $, $g(r,\theta) := \delta_{[\underline{r}, 1]^n \times \mathbb{R}^n}(r, \theta)$, $h(z) := \lambda_2\|z\|_1$, and $c(r, \theta) := {\bf sin}(p\theta/2)$, where $\gamma: \mathbb{R}\to \mathbb{R}$ is defined as $\gamma(t) := 1/t$ if $t\ge \underline{r}$ and $\gamma(t) := -\underline{r}^{-2}(t- \underline{r})+ 1/\underline{r}$ otherwise.

Here, we aim at studying the global complexity of Algorithm~\ref{alg1} for finding an $(\epsilon_1,\epsilon_2,0)$-stationary point (see Definition~\ref{def:ed-stationary}), and establishing subsequential convergence along a {\em constructible} subsequence to a stationary point. We start with the following proposition.
\begin{proposition} \label{complexity_lip}
   Suppose that Assumptions~\ref{assumption 1} and \ref{asm_lip_h} hold, and $\{\beta_t\}$ is chosen to satisfy Assumption~\ref{assumption_beta} with $\delta\in(0,1)$. Let $\{x^t\}$, $\{y^t\}$ and $\{\mu_t\}$ be generated by Algorithm~\ref{alg1}. Then, the following inequalities hold for all $T \ge 1$,
        \begin{align}
            &\frac{1}{T} \sum_{t=1}^T\frac{1}{\mu_t^2}\|x^{t+1}-x^t\|^2 \le \frac{4\rho^{-1}LK_0}{T+1}+ \frac{4\rho^{-1}(L_cM_0+ M_c^2)K_0\gamma_0}{(T+1)^{1-\delta}}, \label{conclusion_lip_1} \\
            &\frac{1}{T} \sum_{t=1}^T \frac{1}{\mu_t}\| x^{t+1} -x^t\|^2 \le \frac{2K_0}{T}, \quad \frac{1}{T} \sum_{t=1}^T \|x^{t+1} -x^t\|^2 \le \frac{2\mu_{\rm max}K_0}{T}, \label{conclusion_lip_2}
        \end{align}
        and
        \begin{align}
            \frac{1}{T}\sum_{t=1}^T\|c(x^{t+1}) -y^t\| \le \frac{2M_h}{\alpha_0(1-\delta)}\frac{1}{(T+1)^{\delta}}+ \sqrt{\frac{8K_0}{\alpha_0(1-\delta)}\frac{1}{(T+1)^{1+\delta}}}, \label{con_lip_3}
        \end{align}
        where $L$, $L_c$, $M_c$ are defined in Assumption~\ref{assumption 1}, $\alpha_0$, $\gamma_0$, $\delta$ are defined in Assumption~\ref{assumption_beta}, $M_h$ is defined in Assumption~\ref{asm_lip_h}, $\rho$ and $\mu_{\max}$ are specified in Algorithm~\ref{alg1},  $M_0$ is defined as in \eqref{M0} and 
        \begin{equation}\label{def_K0}
        K_0 := f(x^1)+g(x^1)+ \frac{\beta_0}{2}\|c(x^1) -y^0\|^2+h(y^0)+\frac{\gamma_0(1+\delta)M_h^2} {2\alpha_0\beta_0} - \inf\{f + g\} > 0.
        \end{equation}
\end{proposition}
\begin{proof}
   From Assumption~\ref{assumption_beta}(i), we have for all $t\ge 1$,
   \begin{align}\label{Hahahahahaha}
       \beta_t &\!\le\! \gamma_0(t+1)^{\delta}  \!\overset{(a)}{\le}\! \gamma_0 t^{\delta}+ \gamma_0\delta t^{\delta -1}
       \!\overset{(b)}{\le}\! \gamma_0 t^{\delta}+ \gamma_0 \delta t^{\delta}  \!=\! \frac{\gamma_0(1+\delta)}{\alpha_0}\alpha_0 t^{\delta} \!\le\! \frac{\gamma_0(1+\delta)}{\alpha_0} \beta_{t-1},
   \end{align}
   where the first and the last inequalities hold because of Assumption~\ref{assumption_beta}(i), $(a)$ holds because $x \mapsto x^{\delta}$ is concave for $\delta \in (0,1)$ and we applied the supergradient inequality, and $(b)$ holds because $t\ge 1$.
   Then, we have 
   \begin{align}
       \frac{\beta_t - \beta_{t-1}}{2\beta_{t-1}^2} &= \frac{\gamma_0(1+\delta)}{2\alpha_0} \frac{\alpha_0}{\gamma_0(1+\delta)} \frac{\beta_t - \beta_{t-1}}{\beta_{t-1}^2} \nonumber\\
       &\le 
       \frac{\gamma_0(1+\delta)}{2\alpha_0}\frac{\beta_t - \beta_{t-1}}{\beta_t \beta_{t-1}} =\frac{\gamma_0(1+\delta)}{2\alpha_0} \left( \frac{1}{\beta_{t-1}}  -\frac{1}{\beta_t}\right),\label{beta_t_beta_t-1}
   \end{align}
   where the inequality follows from \eqref{Hahahahahaha}.

   Applying Lemma~\ref{lemma_H} with $y = y^{t-1}$, we see that for all $t \ge 1$, 
   \begin{align}
       &H(x^{t+1},\beta_t,y^t) 
        \le H(x^t,\beta_{t-1},y^{t-1}) -\frac{1}{2\mu_t}\|x^{t+1}-  x^t\|^2 + \frac{\beta_t - \beta_{t-1}}{2}\|c(x^t) - y^t\|^2 \nonumber \\
       &\overset{(a)}{\le} H(x^t,\beta_{t-1},y^{t-1}) -\frac{1}{2\mu_t}\|x^{t+1} -x^t\|^2 + \frac{\beta_t -\beta_{t-1}}{2}\frac{M_h^2}{\beta_{t-1}^2} \nonumber \\
       &\overset{(b)}{\le} H(x^t,\beta_{t-1},y^{t-1}) -\frac{1}{2\mu_t}\|x^{t+1} -x^t\|^2 + \frac{\gamma_0(1+\delta)M_h^2}{2\alpha_0} \left( \frac{1}{\beta_{t-1}}  -\frac{1}{\beta_t}\right), \label{Whoa}
   \end{align}
    where $(a)$ holds because of the fact that $\beta_{t-1} \left(c(x^t) -y^t \right) \in \partial h(y^t)$ (see \eqref{yh_optcon}) and Assumption~\ref{asm_lip_h} (which implies the boundedness of $\{\partial h(y^{t})\}$), and $(b)$ follows from \eqref{beta_t_beta_t-1}.
   Define
   \begin{equation}\label{tildeH_haha}
        \tilde{H}(x, \beta,y) := H(x,\beta,y)+\frac{\gamma_0(1+\delta)M_h^2}{2\alpha_0} \frac{1}{\beta}.
   \end{equation}
   Then, we see from \eqref{Whoa} that for all $t\ge 1$,
   \[
        \tilde{H}(x^{t+1}, \beta_t, y^t) \le \tilde{H}(x^t, \beta_{t-1}, y^{t-1}) - \frac{1}{2\mu_t}\|x^{t+1} -x^t\|^2.
   \]
   Summing both sides of the above inequality from $t=1$ to $T$, we obtain
  \begin{align}
      \sum_{t =1}^{T} \frac{1}{2\mu_t} \|x^{t+1} - x^t\|^2 &\le \sum_{t=1}^T \left[ \tilde{H}(x^t, \beta_{t-1}, y^{t-1}) - \tilde{H}(x^{t+1}, \beta_t, y^t) \right] \notag\\
      &\overset{(a)}{\le} \tilde{H}(x^1, \beta_0, y^0) - \inf\{f + g\} = K_0, \label{bd_mu_x}
  \end{align}
   where $(a)$ follows from the definition of $\tilde{H}$ in \eqref{tildeH_haha} and the fact that $\frac{\beta_T}{2}\|c(x^{T+1}) -y^T\|^2$ and $h(y^T)$ are nonnegative for all $T\ge 1$, and we used the definition of $K_0$ in \eqref{def_K0} for the equality. Combining \eqref{bd_mu_x} with the observation that $\mu_t\le \mu_{\max}$ for all $t$ proves the inequalities in \eqref{conclusion_lip_2}.
    
    Next, notice that for all $T\ge 1$,
    \begin{align*}
        &\frac{1}{T}\sum_{t=1}^T \frac{1}{\mu_t^2}\|x^{t+1} -x^t\|^2 \overset{(a)}{\le} \frac{1}{T}\sum_{t=1}^T\left(\rho^{-1}L+\rho^{-1}(L_cM_0+ M_c^2)\beta_t \right) \frac{1}{\mu_t}\|x^{t+1} -x^t\|^2  \\
        &\overset{(b)}{\le} \frac{1}{T}\rho^{-1}L\sum_{t=1}^T\frac{1}{\mu_t}\|x^{t+1} -x^t\|^2+ \frac{\rho^{-1}(L_cM_0+ M_c^2)\beta_T}{T}\sum_{t=1}^T\frac{1}{\mu_t}\|x^{t+1} -x^t\|^2  \\
        &\overset{(c)}{\le} \frac{2\rho^{-1}LK_0}{T}+ \frac{2\rho^{-1}(L_cM_0+ M_c^2)K_0\beta_T}{T} \overset{(d)}{\le} \frac{2\rho^{-1}LK_0}{T}+ \frac{2\rho^{-1}(L_cM_0+ M_c^2)K_0\gamma_0(T+1)^{\delta}}{T}  \\
        &\le \frac{4\rho^{-1}LK_0}{T+1}+ \frac{4\rho^{-1}(L_cM_0+ M_c^2)K_0\gamma_0}{(T+1)^{1-\delta}}, 
    \end{align*}
    where $(a)$ holds because of \eqref{bd_1/mu_t}, $(b)$ holds because $\{\beta_t\}$ is nondecreasing, $(c)$ holds because of \eqref{bd_mu_x}, $(d)$ holds because of Assumption~\ref{assumption_beta}(i), and we used the fact that $2T\ge T+1$ in the last inequality. This proves \eqref{conclusion_lip_1}.
    
    Finally, notice that for all $T\ge 1$, 
    \begin{align}\label{hahabeta}
        \sum_{t=1}^T\frac1{\beta_t}\le \sum_{t=1}^T\frac1{\beta_{t-1}} 
        \overset{(a)}\le \sum_{t=1}^T\frac1{\alpha_0 t^\delta} \overset{(b)}\le \frac1{\alpha_0} + \frac1{\alpha_0} \sum_{t=2}^T\int_{t-1}^{t}\frac1{s^\delta}ds \le \frac1{\alpha_0(1-\delta)}T^{1-\delta}.
    \end{align}
    where $(a)$ follows from Assumption~\ref{assumption_beta}(i) and $(b)$ holds as $s\mapsto s^{-\delta}$ is decreasing when $s > 0$. Thus,
    \begin{align*}
        &\frac{1}{T}\sum_{t=1}^T \|c(x^{t+1}) -y^t\|  \le\frac{1}{T} \sum_{t=1}^T  \|c(x^t) -y^t\| + \frac{1}{T}\sum_{t=1}^T \|c(x^{t+1}) -c(x^t)\|  \nonumber \\
        &\overset{(a)}{\le} \frac{1}{T} \sum_{t=1}^T \frac{M_h}{\beta_{t-1}} \!+\! \frac{1}{T} \sum_{t=1}^T \sqrt{\frac{1}{\mu_t\beta_t}}\|x^{t+1} \!-\! x^t\|  \!\overset{(b)}{\le}\! \frac{1}{T} \sum_{t=1}^T \frac{M_h}{\beta_{t-1}} \!+\! \frac{1}{T}\sqrt{\sum_{t=1}^T\frac{1}{\beta_t}}\sqrt{\sum_{t=1}^T\frac{1}{\mu_t}\|x^{t+1} \!-\! x^t\|^2}  \nonumber \\
        &\overset{(c)}{\le}  \frac{M_h}{\alpha_0(1-\delta)T^\delta} + \frac{1}{T}\sqrt{\frac{2K_0}{\alpha_0(1-\delta)}T^{1-\delta}} \le \frac{2M_h}{\alpha_0(1-\delta)}\frac{1}{(T+1)^{\delta}}+ \sqrt{\frac{8K_0}{\alpha_0(1-\delta)}\frac{1}{(T+1)^{1+\delta}}},  
    \end{align*}
    where $(a)$ holds because of Condition~\ref{backtrack condition}(ii), $\beta_{t-1}(c(x^t) -y^t) \in \partial h(y^{t})$ (see \eqref{get_uk_line}) and Assumption~\ref{asm_lip_h} (which implies the boundedness of $\{\partial h(y^{t})\}$), $(b)$ follows from and the Cauchy-Schwarz inequality, $(c)$ follows from  \eqref{hahabeta} and \eqref{bd_mu_x}, and the last inequality holds because $2T\ge T + 1$. This proves \eqref{con_lip_3}.
\end{proof}

Now based on Proposition~\ref{complexity_lip}, we have the following results on complexity and subsequential convergence.
\begin{theorem}[Iteration complexity and subsequential convergence]\label{Thm:alg1}
 Suppose that Assumptions~\ref{assumption 1} and \ref{asm_lip_h} hold, and $\{\beta_t\}$ is chosen to satisfy Assumption~\ref{assumption_beta} with $\delta\in(0,1)$. Let $\{x^t\}$ and $\{y^t\}$ be generated by Algorithm~\ref{alg1}. Then, the following statements hold.
 \begin{enumerate}[{\rm (i)}]
     \item It holds that for all $T \ge 1$, 
      \begin{align}
            &\!\!\!\!\frac{1}{T}\sum_{t=1}^T{\rm dist}^2\left(0, \nabla f(x^{t+1}) + \partial g(x^{t+1})+ J_c(x^{t+1})^T\partial h(y^t)\right)  \nonumber  \\
            &\!\!\!\!\le\! \frac{6\mu_{\max}K_0\lambda_1^2}{T} \!+\! \frac{48\rho^{-1}LK_0}{T\!+\!1} \!+\! \frac{48\rho^{-1}(L_cM_0\!+\! M_c^2)K_0\gamma_0}{(T\!+\!1)^{1-\delta}} \!+\! \frac{12M_h^2M_c^2\eta_0^2}{\alpha_0^2(T\!+\!1)} \!=:\! \Upsilon_T, \label{coro_average} \\
            &\!\!\!\!\ \min\limits_{1\le t\le T} \left\{ {\rm dist}^2\left(0, \nabla f(x^{t+1}) + \partial g(x^{t+1})+ J_c(x^{t+1})^T\partial h(y^t)\right) + \|c(x^{t+1}) - y^t\| \right\}  \nonumber \\
            &\!\!\!\!\le \Upsilon_T + \frac{2M_h}{\alpha_0(1-\delta)}\frac{1}{(T+1)^{\delta}}+ \sqrt{\frac{8K_0}{\alpha_0(1-\delta)}\frac{1}{(T+1)^{1+\delta}}} ,\label{coro_lip_1}
        \end{align}
        where $\lambda_1 := L+ \frac{2^{\delta}\gamma_0}{\alpha_0}M_hL_c$, $\eta_0$ is defined in Assumption~\ref{assumption_beta} and other constants are the same as those in Proposition~\ref{complexity_lip}.
    \item  Define $b_T := \frac{1}{T}\sum_{k=1}^T\|x^{k+1} - x^k\|^2$ for each $T \ge 1$. Then, for any subsequence $\{b_{T_k}\} \subseteq \{b_T\}$ satisfying $b_{T_{k}} \le b_{T_k-1}$ and $T_k >1$,\footnote{In view of Remark~\ref{construct_sub_sq}, we can construct the subsequence $\{b_{T_k}\}$.} and any accumulation point $x^*$ of $\{x^{T_k}\}$, it holds that 
        \begin{equation}\label{optcon1}
            0 \in \nabla f(x^*)+ \partial g(x^*)+ J_c(x^*)^T\partial h(c(x^*)).
        \end{equation}
 \end{enumerate}
\end{theorem}
\begin{proof}
    Using \eqref{yh_optcon}, we see that for each $t\ge 1$,
    \begin{align}
        & {\rm dist} \left(0, \nabla f(x^{t+1})+ \partial g(x^{t+1})+ J_c(x^{t+1})^T\partial h(y^t) \right)  \nonumber \\
        &\le \left\| \nabla f(x^{t+1}) \!-\! \nabla f(x^t) \!-\! \beta_t J_c(x^t)^T(c(x^t) \!-\!y^t)\!-\! \frac{2}{\mu_t}(x^{t+1} \!-\!x^t)\!+\! \beta_{t-1}J_c(x^{t+1})^T(c(x^t) \!-\! y^t) \right\| \nonumber \\
        &\le \| \nabla f(x^{t+1}) - \nabla f(x^t)\| + \frac{2}{\mu_t}\| x^{t+1}-x^t \|+ \| 
        (\beta_{t-1}J_c(x^{t+1}) -\beta_tJ_c(x^t))^T(c(x^t) -y^t)\| \nonumber \\
        & =\left\| \beta_t(J_c(x^{t+1}) -J_c(x^t))^T(c(x^t) -y^t) -(\beta_t - \beta_{t-1})J_c(x^{t+1})^T(c(x^t) -y^t) \right\|  \nonumber \\
        &\ \ \ \  +\|\nabla f(x^{t+1}) - \nabla f(x^t)\| + \frac{2}{\mu_t}\| x^{t+1}-x^t \| \nonumber \\
        &\le  \beta_t\| J_c(x^{t+1}) - J_c(x^t)\| \|c(x^t) - y^t\|+ (\beta_t - \beta_{t-1}) \|J_c(x^{t+1})\| \|c(x^t) -y^t\| \nonumber \\
        &\ \ \ \ +L\|x^{t+1} -x^t\| + \frac{2}{\mu_t}\|x^{t+1} -x^t\| \nonumber \\
        &\overset{(a)}{\le} L\|x^{t+1} -x^t\| + \frac{2}{\mu_t}\|x^{t+1} -x^t\| + \frac{\beta_t}{\beta_{t-1}} M_hL_c\|x^{t+1} - x^t\|+ \frac{\beta_t - \beta_{t-1}}{\beta_{t-1}}M_hM_c \nonumber \\
        &\overset{(b)}{\le} L\|x^{t+1} -x^t\| + \frac{2}{\mu_t}\|x^{t+1} -x^t\| + \frac{\gamma_0}{\alpha_0} \frac{(t+1)^{\delta}}{t^{\delta}}M_hL_c\|x^{t+1} - x^t\|+ \frac{\beta_t -  \beta_{t-1}}{\beta_{t-1}}M_hM_c \nonumber \\
        &\overset{(c)}{\le} L\|x^{t+1} -x^t\| + \frac{2}{\mu_t}\|x^{t+1} -x^t\| + \frac{2^{\delta}\gamma_0}{\alpha_0}M_hL_c\|x^{t+1} - x^t\|+ \frac{\beta_t -  \beta_{t-1}}{\beta_{t-1}}M_hM_c   \nonumber \\
        &= \bigg(L+ \frac{2^{\delta}\gamma_0}{\alpha_0}M_hL_c \bigg)\|x^{t+1} -x^t\|+ \frac{2}{\mu_t}\|x^{t+1} -x^t\|+ \frac{\beta_t -  \beta_{t-1}}{\beta_{t-1}}M_hM_c, \nonumber
    \end{align}
    where $(a)$ holds because of the second relation in \eqref{yh_optcon}, Assumption~\ref{asm_lip_h} (which implies the boundedness of $\{\partial h(y^t)\}$) and Assumption~\ref{assumption 1}(ii), we used Assumption~\ref{assumption_beta}(i) in $(b)$, and we used $2t \ge t+1$ in $(c)$.

    Therefore, upon writing $\lambda_1 := L+ \frac{2^{\delta}\gamma_0}{\alpha_0}M_hL_c$, we have 
    \begin{align}
         &\frac{1}{T} \sum_{t=1}^T{\rm dist}^2 \left(0, \nabla f(x^{t+1})+ \partial g(x^{t+1})+ J_c(x^{t+1})^T\partial h(y^t) \right)  \nonumber \\
        &\le \frac{1}{T}\sum_{t=1}^T\left[ \lambda_1\|x^{t+1} -x^t\|+ \frac{2}{\mu_t}\|x^{t+1} -x^t\|+ \frac{\beta_t -  \beta_{t-1}}{\beta_{t-1}}M_hM_c \right]^2 \nonumber \\
        &\le \frac{3}{T}\sum_{t=1}^T \left[ \lambda_1^2\|x^{t+1} -x^t\|^2+ \frac{4}{\mu_t^2} \|x^{t+1} -x^t\|^2+ \left(\frac{\beta_t -  \beta_{t-1}}{\beta_{t-1}}M_hM_c \right)^2 \right] \nonumber \\
        &= \frac{3\lambda_1^2}{T}\sum_{t=1}^T  \|x^{t+1} -x^t\|^2+\frac{12}{T}\sum_{t=1}^T \frac{1}{\mu_t^2} \|x^{t+1} -x^t\|^2+ \frac{3}{T}\sum_{t=1}^T\left(\frac{\beta_t -  \beta_{t-1}}{\beta_{t-1}}M_hM_c \right)^2 \nonumber \\
        &\overset{(a)}{\le}  \frac{3\lambda_1^2}{T}\sum_{t=1}^T  \|x^{t+1} -x^t\|^2+\frac{12}{T}\sum_{t=1}^T \frac{1}{\mu_t^2} \|x^{t+1} -x^t\|^2+\frac{12M_h^2M_c^2\eta_0^2 /\alpha_0^2}{T+1}, \label{average_dist}
    \end{align}
    where $(a)$ holds because $\frac3{T}\le \frac6{T+1}$ and we have from Assumption~\ref{assumption_beta} that
    \begin{align*}
        \sum_{t =1}^T\left(\frac{\beta_t -\beta_{t-1}}{\beta_{t-1}} \right)^2 \le \sum_{t=1}^T \left( \frac{\eta_0 t^{\delta-1}}{\alpha_0t^{\delta}}\right)^2=\sum_{t=1}^T\frac{\eta_0^2}{\alpha_0^2t^2} \le \frac{\eta_0^2}{\alpha_0^2}\left(1 + \sum_{t=2}^T\frac{1}{t(t-1)}\right) \le \frac{2\eta_0^2}{\alpha_0^2}.
    \end{align*}
Combining \eqref{average_dist} and Proposition~\ref{complexity_lip} proves \eqref{coro_average}.

Next, observe that
    \begin{align}
        &\min_{1\le t \le T} {\rm dist}^2 \left(0, \nabla f(x^{t+1})+ \partial g(x^{t+1})+ J_c(x^{t+1})^T\partial h(y^t) \right) + \| c(x^{t+1}) - y^t\|  \nonumber \\
        &\le\frac{1}{T} \sum_{t=1}^T{\rm dist}^2 \left(0, \nabla f(x^{t+1})+ \partial g(x^{t+1})+ J_c(x^{t+1})^T\partial h(y^t) \right) + \frac{1}{T}\sum_{t=1}^T \|c(x^{t+1}) -y^t\| \nonumber \\
        & \le\frac{6\mu_{\max}K_0\lambda_1^2}{T}+ \frac{48\rho^{-1}LK_0}{T+1}+ \frac{48\rho^{-1}(L_cM_0+ M_c^2)K_0\gamma_0}{(T+1)^{1-\delta}}+\frac{12M_h^2M_c^2\eta_0^2 /\alpha_0^2}{T+1} \nonumber \\
        &\ \ \ \ + \frac{2M_h}{\alpha_0(1-\delta)}\frac{1}{(T+1)^{\delta}}+ \sqrt{\frac{8K_0}{\alpha_0(1-\delta)}\frac{1}{(T+1)^{1+\delta}}}, \nonumber
    \end{align}
    where the last inequality holds because of \eqref{coro_average} and Proposition~\ref{complexity_lip}.
    This proves \eqref{coro_lip_1} and hence establishes item (i).

We now prove item (ii).  We first fix a subsequence $\{b_{T_k}\} \subseteq \{b_T\}$ satisfying $b_{T_{k}} \le b_{T_k-1}$ and $T_k >1$ and an accumulation point $x^*$ of $\{x^{T_k}\}$. In view of Remark~\ref{construct_sub_sq} (where we let $a_t:= \mu_t^{-2}\|x^{t+1}-x^t\|^2$) and \eqref{conclusion_lip_1}, we have
    \begin{align*}
        \frac{1}{\mu_{T_k}^2} \|x^{T_k+1} -x^{T_k}\|^2 \le \frac{4\rho^{-1}LK_0}{T_k}+ \frac{4\rho^{-1}(L_cM_0+ M_c^2)K_0\gamma_0}{T_k^{1-\delta}}. 
    \end{align*}
In addition, considering a convergent subsequence of $\{x^{T_k}\}$ with limit $x^*$ and the bound \eqref{bd_xt-yt}, we may assume upon passing to a further subsequence if necessary that $(x^{T_k},y^{T_k})\to (x^*,y^*)$ for some $y^*$.
Now, since $\beta_{T_k - 1}\to \infty$, we deduce from $\beta_{T_k-1}(c(x^{T_k}) - y^{T_k}) \in \partial h(y^{T_k})$ (see \eqref{get_uk_line}) and Assumption~\ref{asm_lip_h} (which implies the boundedness of $\{\partial h(y^{T_k})\}$) that
\begin{align*}
    c(x^*) = y^*.
\end{align*}
Noting that \eqref{CQ} holds because $\partial^\infty h(c(x^*)) = \{0\}$ (thanks to Assumption~\ref{asm_lip_h}),
we see from the above two displays and Theorem~\ref{convergence} that \eqref{optcon1} holds.
\end{proof}

\begin{remark}[Suggested choice of $\delta$] \label{delta_lip}
Notice that in Theorem~\ref{Thm:alg1}, the $\{\beta_t\}$ in Algorithm~\ref{alg1} is chosen to satisfy Assumption~\ref{assumption_beta} with $\delta\in(0,1)$.
Here, we discuss one possible way of choosing $\delta$.  Notice that
         \begin{align}
        &\min_{1\le t \le T} {\rm dist}^{2\delta} \left(0, \nabla f(x^{t+1})+ \partial g(x^{t+1})+ J_c(x^{t+1})^T\partial h(y^t) \right)+ \| c(x^{t+1}) -y^t\|^{1-\delta}  \nonumber \\
        &\le \frac{1}{T}\sum_{t=1}^T{\rm dist}^{2\delta} \left(0, \nabla f(x^{t+1})+ \partial g(x^{t+1})+ J_c(x^{t+1})^T\partial h(y^t) \right)  + \frac{1}{T}\sum_{t=1}^T \|c(x^{t+1}) -y^t\|^{1-\delta} \nonumber \\
        &\overset{(a)}{\le} \left( \frac{1}{T}\sum_{t=1}^T{\rm dist}^2\! \left(0, \nabla f(x^{t+1})\!+\! \partial g(x^{t+1})\!+\! J_c(x^{t+1})^T\partial h(y^t) \right) \right)^{\delta}\!\!\!+\! \left(\frac{1}{T}\!\sum_{t=1}^T \!\|c(x^{t+1}) \!-\! y^t\| \right)^{1-\delta}\nonumber\\
        & \overset{(b)}{=} \mathcal{O} (T^{-\delta+\delta^2} ), \nonumber
    \end{align}
    where  $(a)$ holds because $x\mapsto x^\theta$ is concave for $\theta \in (0,1)$, and $(b)$ holds because of \eqref{con_lip_3} and \eqref{coro_average}.\footnote{Note that we have from \eqref{con_lip_3} that $\left(\frac{1}{T}\!\sum_{t=1}^T \!\|c(x^{t+1}) \!-\! y^t\| \right)^{1-\delta} = {\cal O}((1-\delta)^{-(1-\delta)}T^{-\delta(1-\delta)}) = {\cal O}(T^{-\delta(1-\delta)})$, where the second equality holds because $\sup_{t\in (0,1)}t^{-t} = \exp(-\inf_{t\in (0,1)}t\ln t) < \infty$.}
 Indeed, both of the two summands in the third line of the above display are of the order $\mathcal{O} (T^{-\delta+\delta^2} )$. Therefore, to obtain an $(\epsilon_1,\epsilon_2,0)$-stationary point, we may choose $\delta$ so that $\epsilon_1^{2\delta} = \epsilon_2^{1-\delta}$, i.e.,
  \begin{equation}\label{choice_delta}
        \delta =  \frac{\ln(\epsilon_2^{-1})}{2 \ln(\epsilon_1^{-1}) + \ln(\epsilon_2^{-1})}\in (0,1).
  \end{equation}
  Then, from \eqref{con_lip_3} and \eqref{coro_average}, the iteration complexity to find an $(\epsilon_1, \epsilon_2, 0)$-stationary point becomes
 $\mathcal{O}\big(\epsilon_1^{-{2}/{(1-\delta)}}+ \epsilon_2^{-{1}/{\delta}} \big) = \mathcal{O}\big(\epsilon_1^{-2} \epsilon_2^{-1} \big)$.
\end{remark}

\begin{remark}
    The recent work \cite{bohm2021variable} considered \eqref{problem} with $g \equiv 0$, $f$ being Lipschitz differentiable, $h$ being Lipschitz continuous and weakly convex, and $c$ being linear, and proposed a variable smoothing method; see \cite[Algorithm~1]{bohm2021variable}. The smoothing parameter in the Moreau envelope of $h$ was chosen in the order of $t^{-1/3}$ in their algorithm. Note that their smoothing parameter plays a role similar to our $\{\beta_t^{-1}\}$ in Algorithm~\ref{alg1}. We can recover their choice of parameter by setting $\epsilon_1 = \epsilon_2$ in \eqref{choice_delta}, which yields $\delta = 1/3$.
\end{remark}

\subsection{Convergence analysis when $\dom h = \mathbb{R}^m$} \label{sec:nonlip}

In this subsection, we consider possibly non-Lipschitz $h$ but require that it be defined everywhere. Specifically, we consider the following assumption.
 \begin{assumption} \label{assumption 3}
    In \eqref{problem}, we have $\dom h = \mathbb{R}^m$ and $\dom g$ being compact.
\end{assumption}
Assumptions~\ref{assumption 1} and \ref{assumption 3} hold in many practical applications such as outlier-robust estimation problem (see, e.g., \cite{peng2023convergence,collins1976robust}). Here, we present the following example, which is a variant of the problem studied in \cite{peng2023convergence} with the residual function based on an $L$-layer multilayer perceptron (MLP) (see, e.g., \cite[Chapter 6]{goodfellow2016deep}).  
\begin{example}\label{MLP}
Let $\{d_i\}_{i=1}^m$ be a dataset with $d_i = (a_i, y_i)$, where $a_i \in \mathbb{R}^{n_0}$ is the feature vector and $y_i$ is a target label or value. 
We define an $L$-layer MLP ($L>2$) as follows: let the model parameter $v := (W_1, b_1, \ldots, W_L,  b_L)$, where 
\begin{align*}
    & W_1 \in \mathbb{R}^{n_1\times n_0}, b_1 \in  \mathbb{R}^{n_1},  W_L \in \mathbb{R}^{1\times n_{L-1}}, b_L \in \mathbb{R}, \\
    &W_{l} \in \mathbb{R}^{n_l \times n_{l-1}}, b_{l} \in \mathbb{R}^{n_l} \ \text{for } l=2,\cdots, L-1,
\end{align*}
for some positive integers $n_1,\ldots,n_{L-1}$.
In addition, let $\sigma: \mathbb{R} \rightarrow \mathbb{R}$ be a smooth scalar activation function. 
Then ${\rm MLP}(a_i;v)$ is defined recursively as follows: 
\[
    z^0 =a_i,\ z^{l} = \sigma\left( W_lz^{l-1}+b_l\right)\ \text{for }l =1,\cdots, L-1,\ \ {\rm MLP}(a_i;v) = W_Lz^{L-1}+b_L,
\]
where for any vector $z$, $\sigma(z)$ is the vector obtained by applying $\sigma$ entrywise to $z$.
Here, we use the hyperbolic tangent or sigmoid functions as activation functions, which are defined as follows respectively:
\[
    \sigma_{\tanh}(u) := \tanh (u) = \frac{e^{u} -e^{-u}}{e^u+ e^{-u}}, \ \ \ \ \sigma_{\rm sigmoid}(u) := \frac{1}{1+ e^{-u}}.
\]
Following \cite[Eq.~(2)]{peng2023convergence} and the discussions following it, we aim to solve the following problem:
\begin{align} \label{MLP_model}
    \min_{v} \frac1m\sum_{i =1}^m \rho\left({\rm MLP}(a_i;v) - y_i\right) + \lambda \|v\|_1, 
\end{align}
where $\rho:\mathbb{R} \rightarrow \mathbb{R}$ is defined as $\rho(\cdot) = |\cdot|^{p}/p$ with $p \in (0,1)$, and $\lambda > 0$; here, the $\ell_1$ regularization is introduced to induce sparsity in the model parameters. One can see that the set of minimizers of the above problem must be contained in 
\begin{equation}\label{def_C}
    C:= \left\{v:\; \|v\|_\infty \le (\lambda m)^{-1}\sum_{i =1}^m \rho\left({\rm MLP}(a_i;0) - y_i\right)\right\}.
\end{equation} 
Then, the problem \eqref{MLP_model} can be seen as an instance of \eqref{problem} satisfying Assumptions~\ref{assumption 1} and \ref{assumption 3}. Specifically, one can take $f \equiv 0$, $g(v) := \delta_{C}(v) + \lambda\|v\|_1 $, $h(u) := \sum_{i=1}^m|u_i|^{p}/p$ and $c_i(v) := {\rm MLP}(a_i;v) -y_i$ for $i = 1,\ldots,m$, and it is known that the proximal mapping of $\gamma h$ (for any $\gamma > 0$) and the projection onto $C$ can be computed efficiently; see, e.g., \cite{GZLHY13} for the proximal mapping of $\gamma h$, and \cite{BergFriedlander08,duchi2008efficient} for the projection onto $C$.
\end{example}

We aim at studying the global complexity of Algorithm~\ref{alg1} to find an $(\epsilon_1,\epsilon_2,\epsilon_3)$-stationary point (see Definition~\ref{def:ed-stationary}) with $\epsilon_3 > 0$, and establish subsequential convergence along a {\em constructible} subsequence to a stationary point. Here, we are considering a more relaxed notion of approximate stationary point by allowing $\epsilon_3 > 0$; this is done to account for the potential lack of Lipschitz continuity of $h$.\footnote{Note that when Assumptions~\ref{assumption 1} and \ref{asm_lip_h} hold, an $(\epsilon_1, \epsilon_2, \epsilon_3)$-stationary point of \eqref{problem} is an $(\epsilon_1+L_cM_h\epsilon_3, \epsilon_2,0)$-stationary point. Indeed, in this case, if $x$ is an $(\epsilon_1, \epsilon_2, \epsilon_3)$-stationary point, then
\begin{align}
    {\rm dist} (0, \nabla f(x)\ +\ &\partial g(x)+ J_c(x)^T\partial h(y)) 
    \le {\rm dist} (0, \nabla f(x)\!+\! \partial g(x)\!+\! J_c(z)^T\partial h(y)) \!+\! \sup_{\xi\in \partial h(y)}\|J_c(x) \!-\! J_c(z)\|\|\xi\| \nonumber \\
    &\le {\rm dist} (0, \nabla f(x)+ \partial g(x)+ J_c(z)^T\partial h(y))+ L_cM_h \|x-z\| \le \epsilon_1+ L_cM_h\epsilon_3, \nonumber
\end{align}
where the second inequality holds because of Assumption~\ref{assumption 1}(ii). Thus, $x$ is an $(\epsilon_1+ L_cM_h\epsilon_3, \epsilon_2,0)$-stationary point.}

We start with the following proposition.

\begin{proposition}\label{theo_xk+1-xk_2}
	Suppose that Assumptions~\ref{assumption 1} and \ref{assumption 3} hold, and $\{\beta_t\}$ is chosen to satisfy Assumption~\ref{assumption_beta} with $\delta\in(0,1)$. Let $\{x^t\}$, $\{y^t\}$ and $\{\mu_t\}$ be generated by Algorithm~\ref{alg1}. Then the following inequalities hold. 
            \begin{align}
            &\frac{1}{T} \sum_{t=1}^T \frac{1}{\mu_t^2}\|x^{t+1} -x^t\|^2  \le  \frac{4\rho^{-1}L\Omega_T}{T+1}+ \frac{4\rho^{-1}(L_cM_0+M_c^2)\gamma_0\Omega_T}{(T+1)^{1-\delta}}\ \ \ \forall T \ge 1, \label{conclu_nonlip_1} \\
            &\frac{1}{T}\sum_{t=1}^T\frac{1}{\mu_t}\|x^{t+1} -x^t\|^2 \le  \frac{4\Omega_T}{T+1}, \ \ \ \ \frac{1}{T}\sum_{t=1}^T \|x^{t+1} -x^t\|^2 \le \frac{4\mu_{\max}\Omega_T}{T+1}\ \ \ \forall T \ge 1, \label{conclu_nonlip_2}\\
            &\|c(x^{t+1}) - y^t\|^2 \le \frac{M_3}{\alpha_0(t+1)^{\delta}}+ \frac{2M_0^2\eta_0}{\alpha_0(t+1)}, \ \ \|c(x^t) -y^t\|^2 \le \frac{2M_3}{\beta_{t-1}} \le \frac{2M_3}{\alpha_0t^{\delta}}\ \ \forall t \ge 1,  \label{bounded_xk_uk}
        \end{align}
         where\footnote{Note that $M_3<\infty$ because $\dom h = \mathbb{R}^m$ and $\dom g$ is bounded by Assumption~\ref{assumption 3} and closed by Assumption~\ref{assumption 1}(iii), and $f$, $g$ and $h$ are continuous in their respective closed domains.}
        \begin{align}
            &\Omega_T :=  M_1+ M_2(\ln T+1), \nonumber \\
            &M_1 :=  f(x^1) + g(x^1)+ \frac{\beta_0}{2}\|c(x^1) - y^0\|^2 +h(y^0) - \inf\{f + g\}, \label{c4}  \\
            &M_2 := \frac{M_3\eta_0}{\alpha_0},\ \ \ \ M_3 := \sup_{x\in \dom g} \{2\left|f(x)+g(x)\right|\}+\sup_{x\in \dom g} \{ h(c(x))\}, \label{c5}
        \end{align}
        $M_0$ is defined in \eqref{M0}, $L$, $L_c$, $M_c$ are defined in Assumption~\ref{assumption 1}, $\alpha_0$, $\eta_0$, $\delta$ and $\gamma_0$ are defined in Assumption~\ref{assumption_beta}, and $\rho$ and $\mu_{\max}$ are specified in Algorithm~\ref{alg1}.
\end{proposition}
\begin{remark}
 Comparing \eqref{conclu_nonlip_1} with \eqref{conclusion_lip_1}, the main difference is the extra log factor in the $\Omega_T$ in \eqref{conclu_nonlip_1} (modulo the difference in the constant factors). This can be attributed to the fact that $h$ is not assumed to be Lipschitz here so that we can only derive $\|c(x^t) - y^t\|^2 = {\cal O}(1/\beta_{t-1})$ (see \eqref{cxt-yt} below). Notice that when $h$ is Lipschitz, one can derive the stronger bound $\|c(x^t) - y^t\| = {\cal O}(1/\beta_{t-1})$, thanks to \eqref{yh_optcon} and the boundedness of $\{\partial h(y^t)\}$.
\end{remark}
\begin{proof}
Applying Lemma~\ref{lemma_H} with $y = c(x^t)$ (which is applicable because $c(x^t) \in \mathbb{R}^m = \dom h$), we see that for all $t\ge 1$,
\begin{align*}
   & H(x^{t+1}, \beta_t, y^t) \le H(x^t, \beta_{t-1}, c(x^t)) - \frac{1}{2\mu_t}\|x^{t+1} -x^t\|^2 + \frac{\beta_t - \beta_{t-1}}{2}\|c(x^t) -y^t\|^2\\
    & \!\le\! H(x^t\!, \beta_{t-1}, c(x^t)) \!+\! \frac{\beta_t \!-\! \beta_{t-1}}{2}\|c(x^t) \!-\! y^t\|^2 
    \!\!\overset{(a)}{=}\!\! f(x^t) \!+\!g(x^t) \!+\! h(c(x^t)) \!+\! \frac{\beta_t \!-\! \beta_{t-1}}{2}\|c(x^t) \!-\! y^t\|^2 \\
    &\overset{(b)}{\le} f(x^t) +g(x^t) + h(c(x^t)) + \frac{\eta_0}{2t^{1-\delta}}\|c(x^t) -y^t\|^2,
\end{align*}
where $(a)$ follows from the definition of $H$ in Lemma~\ref{lemma_H}, and $(b)$ holds thanks to Assumption~\ref{assumption_beta}(ii). Rearranging terms in the above display, we have for all $t \ge 1$,
\begin{align*}
    &\frac{\beta_t}{2}\|c(x^{t+1}) - y^t\|^2 \!\le\! f(x^t) \!+\! g(x^t) \!+\! h(c(x^t)) \!-\! f(x^{t+1}) \!-\! g(x^{t+1}) \!-\! h(y^t) \!+\! \frac{\eta_0}{2t^{1-\delta}}\|c(x^t) -y^t\|^2\\
    &\overset{(a)}{\le} f(x^t)+ g(x^t)+ h(c(x^t)) -f(x^{t+1}) - g(x^{t+1})+ \frac{\eta_0}{2t^{1-\delta}}\|c(x^t) -y^t\|^2 \nonumber \\
    &\overset{(b)}{\le} f(x^t)+ g(x^t)+ h(c(x^t)) -f(x^{t+1}) - g(x^{t+1}) + \frac{M_0^2\eta_0}{2t^{1-\delta}},
\end{align*}
where $(a)$ holds because $h \ge 0$, $(b)$ holds because of \eqref{bd_xt-yt} and the definition of $M_0$ in \eqref{M0}. 
Therefore, for all $t \ge 1$, we obtain that 
\begin{align}
    &\|c(x^{t+1}) - y^t\|^2 \le \frac{2}{\beta_t} \left(f(x^t)+ g(x^t)+ h(c(x^t)) -f(x^{t+1})  -g(x^{t+1}) \right)+ \frac{M_0^2\eta_0}{t^{1-\delta} \beta_t} \nonumber \\
    &\overset{(a)}{\le} \frac{M_3}{\beta_t}+ \frac{M_0^2\eta_0}{t^{1-\delta}\beta_t} \overset{(b)}{\le} \frac{M_3}{\alpha_0(t+1)^{\delta}}+ \frac{M_0^2\eta_0}{t^{1-\delta}\alpha_0(t+1)^{\delta}} \le \frac{M_3}{ \alpha_0(t+1)^{\delta}}+ \frac{2M_0^2\eta_0}{\alpha_0(t+1)}, \nonumber
\end{align}
where $(a)$ holds because of the definition of $M_3$ in \eqref{c5}, $(b)$ holds because of Assumption~\ref{assumption_beta}(i), and the last inequality holds because $2t\ge t+1$ and $\delta\in (0,1)$. 
In addition, for all $t \ge 1$, we have 
\begin{align}
    &\|c(x^t) -y^t\|^2 \le   \| c(x^t) - y^t\|^2 + \frac{2}{\beta_{t-1}}h(y^t) \overset{(a)}{\le} \frac{2h(c(x^t))}{\beta_{t-1}} \overset{(b)}{\le} \frac{2M_3}{\beta_{t-1}}\le \frac{2M_3}{\alpha_0t^{\delta}}, \label{cxt-yt}
\end{align}
where the first inequality holds because $h$ is nonnegative, $(a)$ holds because of \eqref{get_uk_line}, $(b)$ holds because of the definition of $M_3$ in \eqref{c5} and the fact that $x^t \in \dom g$, and the last inequality holds because of Assumption~\ref{assumption_beta}(i). Combining the above two displays, we obtain \eqref{bounded_xk_uk}. 

Next, applying Lemma~\ref{lemma_H} with $y = y^{t-1}$, we obtain upon rearranging terms that for all $t \ge 1$,
\begin{align*}
    \frac{1}{2\mu_t}\|x^{t+1} -x^t\|^2 &\le H(x^t, \beta_{t-1}, y^{t-1}) - H(x^{t+1}, \beta_t, y^t) + \frac{\beta_t- \beta_{t-1}}{2} \|c(x^t) - y^{t}\|^2 \\
    & \le H(x^t, \beta_{t-1}, y^{t-1}) - H(x^{t+1}, \beta_t, y^t) + \frac{\beta_t -\beta_{t-1}}{\alpha_0t^{\delta}}M_3, 
\end{align*}
where the last inequality holds because of \eqref{cxt-yt}.
Summing both sides of the above display from $t =1$ to $T$, we have
\begin{align}
    &\sum_{t = 1}^T \frac{1}{2\mu_t}\|x^{t+1} -x^t\|^2 \le H(x^1, \beta_0, y^0)- H(x^{T+1}, \beta_T, y^T) + \sum_{t=1}^T\frac{\beta_t - \beta_{t-1}}{\alpha_0t^{\delta}}M_3 \nonumber \\
    & \overset{(a)}{\le} H(x^1, \beta_0, y^0) -\inf\{f + g\}  + \sum_{t=1}^T\frac{\beta_t - \beta_{t-1}}{\alpha_0t^{\delta}}M_3  \nonumber  \\
	& \overset{(b)}{\le}  M_1  + \sum_{t=1}^T\frac{\eta_0 t^{\delta-1}}{\alpha_0t^{\delta}}M_3 = M_1  + \frac{\eta_0}{\alpha_0}\sum_{t=1}^T \frac{ M_3}{t}  \le M_1+ \frac{\eta_0M_3}{\alpha_0} (\ln T+1) = M_1+ M_2(\ln T+1), \nonumber
\end{align}
where $(a)$ holds because $\frac{\beta_T}{2}\|c(x^{T+1}) - y^T\|^2\ge 0$ and $h(y^T) \ge 0$, $(b)$ holds because of Assumption~\ref{assumption_beta}(ii) and the definition of $M_1$ in \eqref{c4}, the last inequality holds because $x \mapsto \frac1x$ is decreasing, and we use the definition of $M_2$ in \eqref{c5} for the last equality.
Then, we have 
\begin{align}
    \frac{1}{T}\sum_{t=1}^T\frac{1}{\mu_t}\|x^{t+1} -x^t\|^2 \le\frac{2M_1+2M_2(\ln T+1)}{T}\le  \frac{4M_1+4M_2(\ln T+1)}{T+1}, \label{bd_mu_x_2}
\end{align}
where the last inequality holds because $2T\ge T+1$. The inequalities in \eqref{conclu_nonlip_2} now follow immediately from \eqref{bd_mu_x_2} and the observation that $\mu_{\max} \ge \mu_t$ for all $t\ge 1$.

Next, note that
\begin{align*}
    &\frac{1}{T}\sum_{t=1}^T \frac{1}{\mu_t^2}\|x^{t+1} -x^t\|^2  \le \frac{1}{T}\sum_{t=1}^T \left(\rho^{-1}L+ \rho^{-1}(L_cM_0+ M_c^2)\beta_t \right)\frac{1}{\mu_t}\|x^{t+1} -x^t\|^2  \nonumber \\
    &\le\frac{1}{T}\left(\rho^{-1}L+ \rho^{-1}(L_cM_0+ M_c^2)\beta_T \right)\sum_{t=1}^T\frac{1}{\mu_t}\|x^{t+1} -x^t\|^2 \nonumber \\
    &\le \left(\rho^{-1}L+ \rho^{-1}(L_cM_0+ M_c^2)\beta_T \right)\frac{4M_1+4M_2(\ln T+1)}{T+1} \nonumber \\
    &\le \left(\rho^{-1}L+ \rho^{-1}(L_cM_0+ M_c^2)\gamma_0(T+1)^{\delta} \right)\frac{4M_1+4M_2(\ln T+1)}{T+1}, 
\end{align*}
where the first inequality holds because of \eqref{bd_1/mu_t}, the second inequality holds because $\{\beta_t\}$ is nondecreasing, the third inequality follows from \eqref{bd_mu_x_2} and the last inequality holds because of Assumption~\ref{assumption_beta}(i).
This proves \eqref{conclu_nonlip_1}.
\end{proof}
Now, based on Proposition~\ref{theo_xk+1-xk_2}, we have the following results on complexity and subsequential convergence. 
\begin{theorem}[Iteration complexity and subsequential convergence]
    Suppose that Assumptions~\ref{assumption 1} and \ref{assumption 3} hold, and $\{\beta_t\}$ is chosen to satisfy Assumption~\ref{assumption_beta} with $\delta\in(0,1)$. Let $\{x^t\}$ and $\{y^t\}$ be generated by Algorithm~\ref{alg1}. Then, the following statements hold.
    \begin{enumerate}[\rm (i)]
            \item  It holds that for all $T\ge 1$,
            \begin{align}
                 &\!\!\!\!\!\! \frac{1}{T} \sum_{t=1}^T {\rm dist}^2\left(0, \nabla f(x^{t+1})+ \partial g(x^{t+1})+ J_c(x^t)^T\partial h(y^t)\right)  \nonumber \\
                 &\!\!\!\!\!\le \frac{32(\rho^{-1}L+1)+ 8\mu_{\max}L^2}{T+1}\Omega_T+\frac{32\rho^{-1}(L_cM_0+M_c^2)\gamma_0}{(T+1)^{1-\delta}}\Omega_T+ \frac{16\eta_0M_c^2M_2}{(1-\delta)(T+1)}, \label{aver_sub2} \\ 
                &\!\!\!\!\!\! \min\limits_{1\le t\le T} \!\big\{ {\rm dist}^2\big(0, \nabla f(x^{t+1}) \!+\! \partial g(x^{t+1}) \!+\! J_c(x^t)^T\partial h(y^t)\big) \!+\! \|x^{t+1}\!-\!x^t\|^2 \!+\! \|c(x^{t+1}) \!-\! y^t\|^2 \big\} \nonumber \\
                   &\!\!\!\!\!\le \frac{\lambda_2+ \lambda_3(\ln T+1)}{T+1}+ \frac{\lambda_4(M_1+M_2(\ln T+1))}{(T+1)^{1-\delta}}+ \frac{M_3}{\alpha_0(T+1)^{\delta}} + \frac{2M_0^2\eta_0}{\alpha_0(T+1)}, \label{min_sub2}
            \end{align}
        where
        \begin{align}
                &\Omega_T:= M_1+M_2(\ln T+1),  \nonumber \\
                & \lambda_2 := 32\rho^{-1}LM_1+32M_1+8\mu_{\max}M_1L^2+ \frac{16\eta_0M_c^2M_2}{(1-\delta)}+ 4\mu_{\max}M_1,  \label{lam_2} \\
                &\lambda_3 := 32\rho^{-1}LM_2+32M_2+8\mu_{\max}M_2L^2+ 4\mu_{\max}M_2, \label{lam_3} \\
                &\lambda_4 := 32\rho^{-1}(L_cM_0+M_c^2)\gamma_0, \label{lam_4}
            \end{align}
        and other constants are the same as those in Proposition~\ref{theo_xk+1-xk_2}.

        \item  
        Define $b_T := \frac{1}{T}\sum_{k=1}^T\|x^{k+1} - x^k\|^2$ for each $T \ge 1$. Then, for any subsequence $\{b_{T_k}\} \subseteq \{b_T\}$ satisfying $b_{T_{k}} \le b_{T_k-1}$ and $T_k >1$,\footnote{In view of Remark~\ref{construct_sub_sq}, we can construct the subsequence $\{b_{T_k}\}$.} and any accumulation point $x^*$ of $\{x^{T_k}\}$ such that \eqref{CQ} holds, we have
        \begin{equation}\label{optcon2}
            0 \in \nabla f(x^*)+ \partial g(x^*)+ J_c(x^*)^T\partial h(c(x^*)).
        \end{equation}
    \end{enumerate}
\end{theorem}
\begin{proof}
Using \eqref{yh_optcon}, we see that for each $t\ge 1$,
\begin{align}
    &{\rm dist}(0, \nabla f(x^{t+1})+ \partial g(x^{t+1})+ J_c(x^t)^T\partial h(y^t)) \nonumber \\
    &\le \left\|-\frac{2}{\mu_t}(x^{t+1} - x^t)- \left( \beta_t -\beta_{t-1} \right)J_c(x^t)^T(c(x^t) -y^t) +\nabla f(x^{t+1}) -\nabla f(x^t) \right\|      \nonumber\\
    &\le \frac{2}{\mu_t}\|x^{t+1} -x^t\| + (\beta_t -\beta_{t-1})\|J_c(x^t)^T(c(x^t) -y^t)\| + \| \nabla f(x^{t+1}) -\nabla f(x^t) \| \nonumber \\
    &\overset{(a)}{\le} \frac{2}{\mu_t}\|x^{t+1}-x^t\|+\eta_0t^{\delta -1}M_c\|c(x^t) -y^t\| + L\|x^{t+1} -x^t\|  \nonumber \\
    &= \left( \frac{2}{\mu_t}+ L \right)  \|x^{t+1} -x^t\|+\eta_0t^{\delta -1}M_c\|c(x^t) -y^t\|,  \label{sum_dist_nlp}
\end{align}
where $(a)$ holds because of Assumption~\ref{assumption_beta}(ii) and Assumption~\ref{assumption 1}(i) and (ii).
Notice that
\begin{align}
    &\frac{1}{T}\sum_{t=1}^T \frac{t^{2\delta-2}}{\beta_{t-1}}\le \frac{1}{T} \sum_{t=1}^T \frac{1}{\alpha_0 t^{2-\delta}} = \frac1{\alpha_0T}\left(1 + \sum_{t=2}^T t^{\delta-2}\right)
    \le \frac1{\alpha_0T}\left(1 + \sum_{t=2}^T \int_{t-1}^t s^{\delta-2}ds\right)\nonumber\\
    &
    = \frac1{\alpha_0T}\left(1 + \frac{1-T^{\delta-1}}{1-\delta}\right) \le \frac1{\alpha_0T}\left(1 + \frac{1}{1-\delta}\right)\le \frac2{\alpha_0T(1-\delta)} \le \frac{4}{\alpha_0(1-\delta)(T+1)}, \label{sum_t/bt-1}
\end{align}
where the first inequality follows from Assumption~\ref{assumption_beta}(i), the second inequality holds because $x \mapsto x^{\delta-2}$ is decreasing for $x >0$ since $\delta \in (0,1)$, and the last inequality holds because $2T \ge T+1$.
Hence, from the above display, we can now deduce from \eqref{sum_dist_nlp} that
\begin{align*}
    &\frac{1}{T}\sum_{t=1}^T {\rm dist}^2\left(0, \nabla f(x^{t+1}) + \partial g(x^{t+1})+ J_c(x^t)^T\partial h(c(y^t))\right)\nonumber \\
    &\le \frac{1}{T}\sum_{t=1}^T \left(\left( \frac{2}{\mu_t}+ L \right) \|x^{t+1} -x^t\| + \eta_0t^{\delta -1}M_c\|c(x^t) -y^t\|\right)^2  \nonumber\\
    &\le  \frac{2}{T}\sum_{t=1}^T \left(\frac{2}{\mu_t}+L \right)^2\|x^{t+1} -x^t\|^2 + \frac{2}{T} \sum_{t=1}^T \eta_0^2t^{2\delta -2}M_c^2\|c(x^t) -y^t\|^2\nonumber \\
    & \overset{(a)}{\le} \frac{2}{T}\sum_{t=1}^T \left(\frac{2}{\mu_t}+L \right)^2\|x^{t+1} -x^t\|^2+ \frac{1}{T}\sum_{t=1}^T\frac{4\eta_0^2t^{2\delta -2}M_c^2M_3}{\beta_{t-1}}   \nonumber \\
    &\overset{(b)}{\le} \frac{2}{T}\sum_{t=1}^T \left(\frac{4}{\mu_t^2}+ \frac{4}{\mu_t}+ L^2 \right)\|x^{t+1} -x^t\|^2+\frac{16\eta_0M_c^2M_2}{(1-\delta)(T+1)}, 
\end{align*}
where $(a)$ holds because $\|c(x^t) -y^t\|^2 \le \frac{2M_3}{\beta_{t-1}}$ (see \eqref{bounded_xk_uk}), $(b)$ follows from \eqref{sum_t/bt-1} and the definition of $M_2$ in \eqref{c5}. This result together with Proposition~\ref{theo_xk+1-xk_2} proves \eqref{aver_sub2}.

Therefore, it holds that 
\begin{align}
    &\min_{1\le t\le T} \left\{ {\rm dist}^2\left(0, \nabla f(x^{t+1}) + \partial g(x^{t+1})+ J_c(x^t)^T\partial h(c(y^t))\right) + \|x^{t+1} -x^t\|^2 \right\} \nonumber \\
    &\le \frac{1}{T}\sum_{t=1}^T {\rm dist}^2\left(0, \nabla f(x^{t+1}) + \partial g(x^{t+1})+ J_c(x^t)^T\partial h(c(y^t))\right)+ \frac{1}{T}\sum_{t=1}^T\|x^{t+1} -x^t\|^2 \nonumber \\
    &\overset{(a)}{\le} \frac{32\rho^{-1}L(M_1+M_2(\ln T+1))}{T+1}+ \frac{32\rho^{-1}(L_cM_0+M_c^2)\gamma_0(M_1+M_2(\ln T+1))}{(T+1)^{1-\delta}} \nonumber \\
    &\ \ \ \ + \frac{32M_1+32M_2(\ln T+1)}{T+1}+ \frac{8\mu_{\max}M_1L^2+8\mu_{\max}M_2L^2(\ln T+1)}{T+1} +\frac{16\eta_0M_c^2M_2}{(1-\delta)(T+1)} \nonumber\\
    &\ \ \ \ + \frac{4\mu_{\max}M_1+4\mu_{\max}M_2(\ln T+1)}{T+1} \nonumber \\
    &=\frac{\lambda_2+ \lambda_3(\ln T+1)}{T+1}+ \frac{\lambda_4(M_1+M_2(\ln T+1))}{(T+1)^{1-\delta}},  \nonumber
\end{align}
where $(a)$ holds because of \eqref{aver_sub2} and Proposition~\ref{theo_xk+1-xk_2}, and the last equality follows from the definitions of $\lambda_2$, $\lambda_3$ and $\lambda_4$ in \eqref{lam_2}, \eqref{lam_3} and \eqref{lam_4}, respectively.
This together with \eqref{bounded_xk_uk} proves \eqref{min_sub2} and hence establishes item (i).

We now prove item (ii). We first fix a subsequence $\{b_{T_k}\} \subseteq \{b_T\}$ satisfying $b_{T_{k}} \le b_{T_k-1}$ and $T_k >1$ and an accumulation point $x^*$ of $\{x^{T_k}\}$ such that \eqref{CQ} holds. In view of Remark~\ref{construct_sub_sq} (where we let $a_t := \mu_t^{-2}\|x^{t+1}-x^t\|^2$) and \eqref{conclu_nonlip_1}, we have
\begin{align}
    & \frac{\rho}{4\mu_{T_k}^2} \|x^{T_k+1} -x^{T_k}\|^2 \le \frac{L(M_1+M_2(\ln T_k+1))}{T_k}+ \frac{(L_cM_0+M_c^2)\gamma_0(M_1+M_2(\ln T_k+1))}{T_k^{1-\delta}}. \nonumber
\end{align}
In addition, in view of the boundedness of $\dom g$ and \eqref{bd_xt-yt}, by passing to a further subsequence if necessary, we may assume that $(x^{T_k},y^{T_k})\to (x^*,y^*)$ for some $y^*$. Then, in view of \eqref{bounded_xk_uk}, we have
\begin{align}
    c(x^*) = y^*. \nonumber
\end{align}
Combining the above two displays with the assumption that \eqref{CQ} holds at $x^*$, we see from Theorem~\ref{convergence} that \eqref{optcon2} holds.
\end{proof}

\begin{remark}
  In view of \eqref{min_sub2}, if we choose $\delta=0.5$, the iteration complexity for obtaining an $(\epsilon,\epsilon,\epsilon)$-stationary point is $\tilde{\cal O}(\epsilon^{-4})$. Note that due to the presence of the factor $1-\delta$ in the denominator of the last term in \eqref{aver_sub2}, the techniques in Remark~\ref{delta_lip} cannot be extended to provide insight on an optimal choice of $\delta$, because $\sup_{\delta\in (0,1)} (1-\delta)^{-\delta} = \infty$.
\end{remark}

\subsection{Convergence analysis with bounded convex domains} \label{sec:43}

In this subsection, we consider the following assumption, which allows $\dom h \neq \mathbb{R}^m$.
\begin{assumption} \label{assumption 2}
   In \eqref{problem}, $\dom g$ is compact and $\dom h$ is closed and convex. Moreover, all stationary points of 
   \begin{equation}\label{pro_for_proof}
    \min\limits_{x\in \dom g} \frac{1}{2}{\rm dist}(c(x), \dom h)^2
\end{equation}
are global minimizers.\footnote{Recall that an $x^*$ is a stationary point of \eqref{pro_for_proof} if $0 \in J_c(x^*)^T(c(x^*) - {\rm Proj}_{\dom h}(c(x^*))) + \partial \delta_{\dom g}(x^*)$.}
\end{assumption}
Compared with Assumption~\ref{assumption 3}, Assumption~\ref{assumption 2} allows $h$ with $\dom h\neq \mathbb{R}^m$, but requires further assumptions on the structure of \eqref{pro_for_proof}. Notice that when $\dom h = \mathbb{R}^n$, the objective of \eqref{pro_for_proof} becomes identically zero and hence all stationary points are trivially globally optimal; this latter condition is also implied by each of the PL-type conditions stated in \cite[Assumption~4]{deng2025single}, \cite[Assumption~1(iv)]{LuMeiXiao24} and \cite[Eq.~(A5)]{Alacaoglu24a}, respectively. The next proposition presents other sufficient conditions for all stationary points of \eqref{pro_for_proof} to be globally optimal.
\begin{proposition}[{The desired property of \eqref{pro_for_proof}}]\label{prop_haha}
  Suppose that Assumption~\ref{assumption 1} holds, $\dom g$ and $\dom h$ are convex and $c$ is $-(\dom h)^\infty$-convex. Then every stationary point of \eqref{pro_for_proof} is a global minimizer.
\end{proposition}
\begin{proof}
  It suffices to show that \eqref{pro_for_proof} is a convex optimization problem.  To see this, we first note that
\[
    {\rm hzn}~ \left(\frac{1}{2}{\rm dist}^2(\cdot, \dom h) \right) = \left(\{x:\;{\rm dist}^2(x, \dom h)\le0\}\right)^{\infty} = ( \dom h)^{\infty},
\]
where the first equality follows from \cite[Theorem~8.7]{Rockafellar+1970} and the second equality holds because $\dom h$ is closed (see Assumption~\ref{assumption 1}(iii)).
Hence, using \cite[Theorem~1]{burke2021study} and the assumptions that $\dom h$ is convex and $c$ is $-(\dom h)^\infty$-convex, we deduce that $\frac{1}{2}{\rm dist}^2(c(\cdot), \dom h)$ is a convex function. Since $\dom g$ is also convex by assumption, we see that \eqref{pro_for_proof} is a convex optimization problem.
\end{proof}

We now present an example that satisfies both Assumptions~\ref{assumption 1} and \ref{assumption 2}. 
\begin{example}\label{QCQP}
    Consider a variant of the penalized quadratically constrained quadratic program (QCQP) from \cite{boob2025level} with $\ell_p$ penalty function, $p \in (0,1)$, which is shown as follows:
    \begin{equation} \label{QCQP_pro}
        \begin{array}{cl}
             \min\limits_{x\in \mathbb{R}^n}& \frac{1}{2} x^TQ_0x+ b_0^Tx+ \alpha \|x\|^p_p  \\
             {\rm s.t.}& \frac{1}{2} x^TQ_ix+ b_i^Tx+ r_i \le 0, i=1,2,\cdots, m, \\
                       & \|x\|_\infty \le r,
        \end{array}
    \end{equation}
    where each $Q_i$, $i=0,\ldots,m$, is an $n\times n$ matrix, $b_i\in \mathbb{R}^n$ and $r_i<0$ for all $i$, $\alpha>0$, $r>0$, and each $Q_i$, $i = 1,\ldots,m$, is positive semidefinite. Notice that this is an instance of \eqref{problem} satisfying Assumptions~\ref{assumption 1} and \ref{assumption 2}.  Specifically, one can take $f(x) := \frac{1}{2}x^TQ_0x+ b_0^Tx$, $g(x) := \alpha \| x\|_p^p+ \delta_{\|\cdot\|_\infty \le r}(x)$, $c_i(x) := \frac{1}{2}x^TQ_ix+ b_i^Tx+ r_i$ for $i=1, \ldots, m$ and $h(y):= \delta_{\mathbb{R}_-^m}(y)$. One can see that the proximal mapping of $\gamma g$ (for any $\gamma > 0$) can be computed efficiently (see, e.g., \cite{GZLHY13}). Moreover, $\dom g$ and $\dom h$ are convex. In addition, since $\dom h = \mathbb{R}^m_-$ and each $c_i$ is convex, we see that $c$ is $-(\dom h)^\infty$-convex as well. Thus, every stationary point of \eqref{pro_for_proof} is a global minimizer in view of Proposition~\ref{prop_haha}.
\end{example}
We will derive the iteration complexity for Algorithm~\ref{alg1} (with $\delta\in (0,1/2)$) to generate an $(\epsilon_1, M_0,\epsilon_3)$-stationary point, where $M_0$ was given in \eqref{M0}. Note that this does not imply any vanishing bounds on $\{\|c(x^t) - y^t\|\}$. Nevertheless, we show how to construct a subsequence $\{x^{T_k}\}$ such that, together with the corresponding $\{y^{T_k}\}$, we have $\lim_{k \rightarrow \infty} \|c(x^{T_k}) - y^{T_k}\| =0$, and we also establish subsequential convergence along this subsequence to a stationary point, under a standard constraint qualification. 

We start with the following proposition.

\begin{proposition}\label{theo_xk+1-xk}
    Suppose that Assumptions~\ref{assumption 1} and \ref{assumption 2} hold, and $\{\beta_t\}$ is chosen to satisfy Assumption~\ref{assumption_beta} with $\delta\in (0,\frac12)$. Let $\{x^t\}$, $\{y^t\}$ and $\{\mu_t\}$ be generated by Algorithm~\ref{alg1}. 
    Then the following inequalities hold for all $T\ge 1$.
            \begin{align}
            &\frac{1}{T} \sum_{t=1}^T \frac{1}{\mu_t^2}\|x^{t+1} -x^t\|^2  \le  \frac{4\lambda_5M_1}{(T+1)^{1-\delta}}+ \frac{4\lambda_5M_0^2\gamma_0}{(T+1)^{1-2\delta}}, \label{conclu3_1} \\
            &\frac{1}{T}\sum_{t=1}^T\frac{1}{\mu_t}\|x^{t+1} -x^t\|^2 \le   \frac{4M_1}{T+1}+ \frac{4M_0^2\gamma_0}{(T+1)^{1-\delta}}, \label{conclu3_2}  \\
            &\frac{1}{T}\sum_{t=1}^T \|x^{t+1} -x^t\|^2 \le \frac{4\mu_{\max}M_1}{T+1}+ \frac{4\mu_{\max}M_0^2\gamma_0 }{(T+1)^{1-\delta}}, \label{conclu3_3}
        \end{align}
         where 
        \begin{align}
           \lambda_5 := \rho^{-1}L+ \rho^{-1}(L_cM_0+ M_c^2)\gamma_0, \label{lam_5} 
        \end{align}
        $M_0$ is defined as in \eqref{M0}, $M_1$ is defined as in \eqref{c4}, $L$, $L_c$, $M_c$ are defined in Assumption~\ref{assumption 1}, $\delta$ and $\gamma_0$ are defined in Assumption~\ref{assumption_beta}, and $\rho$ and $\mu_{\max}$ are specified in Algorithm~\ref{alg1}.
\end{proposition}
\begin{proof}
Applying Lemma~\ref{lemma_H} with $y = y^{t-1}$, we have that for all $t\ge 1$,
\begin{align*}
     &H(x^{t+1}, \beta_t,y^t) \le H(x^t, \beta_{t-1},y^{t-1}) -\frac{1}{2\mu_t}\|x^{t+1} - x^t\|^2+ \frac{\beta_t - \beta_{t-1}}{2}\|c(x^t) - y^t\|^2 \\
    &\le H(x^t, \beta_{t-1},y^{t-1}) -\frac{1}{2\mu_t}\|x^{t+1} - x^t\|^2 + \left(\beta_t - \beta_{t-1} \right)M_0^2, \nonumber 
\end{align*}
where the last inequality holds because of \eqref{bd_xt-yt} and the definition of $M_0$ in \eqref{M0}.
Summing both sides of the above display from $t =1$ to $T$, we obtain upon rearranging terms that
\begin{align}
    &\sum_{t = 1}^{T}\frac{1}{2\mu_t}\|x^{t+1} -x^t\|^2 \le   \sum_{t=1}^{T} \left[ H(x^t, \beta_{t-1}, y^{t-1	}) - H(x^{t+1}, \beta_t, y^t)\right] + \sum_{t=1}^{T} (\beta_{t} - \beta_{t-1})M_0^2, \nonumber \\
    & = H(x^1, \beta_0, y^0) - H(x^{T+1}, \beta_T, y^T) + (\beta_T-\beta_0)M_0^2
    \overset{(a)}\le H(x^1, \beta_0, y^0) - \inf\{f+g\} + \beta_TM_0^2, \nonumber
\end{align}
where $(a)$ holds because $\frac{\beta_T}{2}\|c(x^{T+1}) -y^T\|^2\ge0$, $h(y^T)\ge 0$ and $\beta_0 >0$.
Then, we have 
\begin{align}
    &\frac{1}{T}\sum_{t=1}^T \frac{1}{\mu_t}\|x^{t+1} -x^t\|^2   \nonumber \\
    &\le \frac{2M_1}{T}+ \frac{2M_0^2\gamma_0(T+1)^{\delta}}{T} \le \frac{4M_1}{T+1}+ \frac{4M_0^2\gamma_0(T+1)^{\delta}}{T+1} = \frac{4M_1}{T+1}+ \frac{4M_0^2\gamma_0}{(T+1)^{1-\delta}}, \label{sum_nlip_1}
\end{align}
where the first inequality holds because of Assumption~\ref{assumption_beta}(i) and the second inequality holds because $2T\ge T+1$. This proves \eqref{conclu3_2}.
The inequality \eqref{conclu3_3} follows immediately from \eqref{conclu3_2} upon noting that $\mu_t\le \mu_{\max}$ for all $t$.

Next, observe that 
\begin{align*}
    &\frac{1}{T}\sum_{t=1}^T \frac{1}{\mu_t^2}\|x^{t+1} -x^t\|^2  \overset{(a)}{\le} \frac{1}{T}\sum_{t=1}^T \left(\rho^{-1}L+ \rho^{-1}(L_cM_0+ M_c^2)\beta_t \right)\frac{1}{\mu_t}\|x^{t+1} -x^t\|^2  \nonumber \\
    &\overset{(b)}{\le}\frac{1}{T}\left(\rho^{-1}L+ \rho^{-1}(L_cM_0+ M_c^2)\beta_T \right)\sum_{t=1}^T\frac{1}{\mu_t}\|x^{t+1} -x^t\|^2 \nonumber \\
    &\overset{(c)}{\le} \left(\rho^{-1}L+ \rho^{-1}(L_cM_0+ M_c^2)\beta_T \right)\left(\frac{4M_1}{T+1}+ \frac{4M_0^2\gamma_0}{(T+1)^{1-\delta}}\right) \nonumber \\
    &\overset{(d)}{\le} \left(\rho^{-1}L+ \rho^{-1}(L_cM_0+ M_c^2)\gamma_0(T+1)^{\delta} \right) \left( \frac{4M_1}{T+1}+ \frac{4M_0^2\gamma_0}{(T+1)^{1-\delta}} \right) \nonumber \\
    &\le \left(\rho^{-1}L+ \rho^{-1}(L_cM_0+ M_c^2)\gamma_0 \right) (T+1)^{\delta}\left( \frac{4M_1}{T+1}+ \frac{4M_0^2\gamma_0}{(T+1)^{1-\delta}} \right), 
\end{align*}
where $(a)$ holds because of \eqref{bd_1/mu_t}, $(b)$ holds because $\{\beta_t\}$ is nondecreasing, $(c)$ follows from \eqref{sum_nlip_1}, $(d)$ holds because of Assumption~\ref{assumption_beta}(i), the last inequality holds because $T \ge 1$. Recalling the definition of $\lambda_5$ in \eqref{lam_5}, the inequality \eqref{conclu3_1} follows from the above display.
\end{proof}

We can now establish the following results on complexity and subsequential convergence.
\begin{theorem}[Iteration complexity and subsequential convergence]\label{Thm43iteration}
    Suppose that Assumptions~\ref{assumption 1} and \ref{assumption 2} hold, and $\{\beta_t\}$ is chosen to satisfy Assumption~\ref{assumption_beta} with $\delta\in(0,\frac{1}{2})$. Let $\{x^t\}$ and $\{y^t\}$ be generated by Algorithm~\ref{alg1}. Then, the following statements hold.
    \begin{enumerate}[\rm (i)]
            \item  It holds that for all $T\ge 1$,
            \begin{align}
                 & \min\limits_{1\le t\le T}\left\{ {\rm dist}^2\left(0, \nabla f(x^{t+1}) + \partial g(x^{t+1})+ J_c(x^t)^T\partial h(y^t)\right)+ \|x^{t+1}-x^t\|^2\right\} \nonumber \\
                   &\le \frac{\lambda_6}{T+1}+ \frac{\lambda_7}{(T+1)^{1-\delta}}+ \frac{\lambda_8}{(T+1)^{1-2\delta}}, \label{min_sub3}
            \end{align}
        where
        \begin{align}
                &\lambda_6 := (8L^2 + 4)\mu_{\max}M_1+ 32M_1+ \frac{8\eta_0^2M_c^2M_0^2}{1-2\delta}, \label{lam_6} \\
                &\lambda_7 := (8L^2 + 4)\mu_{\max}M_0^2\gamma_0+ 32M_0^2\gamma_0+ 32\lambda_5M_1, \label{lam_7}  \\
                &\lambda_8 := 32\lambda_5M_0^2\gamma_0, \label{lam_8}
            \end{align}
        and other constants are the same as those in Proposition~\ref{theo_xk+1-xk_2}(i).

        \item  
         Define $b_T := \frac{1}{T}\sum_{k=1}^T\|x^{k+1} - x^k\|^2$ for each $T \ge 1$. Then, for any subsequence $\{b_{T_k}\} \subseteq \{b_T\}$ satisfying $b_{T_{k}} \le b_{T_k-1}$ and $T_k >1$,\footnote{In view of Remark~\ref{construct_sub_sq}, we can construct the subsequence $\{b_{T_k}\}$.} it holds that $\lim_{k\to\infty}\|c(x^{T_k}) - y^{T_k}\| = 0$; moreover, any accumulation point $x^*$ of $\{x^{T_k}\}$ satisfies $c(x^*)\in \dom h$, and if, in addition, \eqref{CQ} holds at $x^*$, then
        \begin{equation}\label{optcon3}
            0 \in \nabla f(x^*)+ \partial g(x^*)+ J_c(x^*)^T\partial h(c(x^*)).
        \end{equation}
    \end{enumerate}
\end{theorem}
\begin{remark}[Interpreting \eqref{min_sub3}]\label{rem:explain_cplx}
  Recall that when Assumptions~\ref{assumption 1} and \ref{asm_lip_h} hold, and $g$ is also assumed to be the indicator function of a closed set while $h$ is also assumed to be weakly convex, for the $F^{[t]}$ in \eqref{subsubsub}, we have for all $x\in \dom g$ and all sufficiently large $t$ that
  \[
  \partial F^{[t]}(x) = \nabla f(x) + \partial g(x) + J_c(x)^T\nabla e_{1/\beta_t}h(c(x))\subseteq \nabla f(x) + \partial g(x) + J_c(x)^T\partial h({\rm Prox}_{1/\beta_t}(c(x))),
  \]
  and iteration complexity bound was derived in \cite[Theorem~3.3]{kume2024variable} for finding an $x$ satisfying ${\rm dist}(0,\partial F^{[t]}(x))\le \epsilon$.
  In view of the above display and \eqref{get_uk_line}, we can view \eqref{min_sub3} as an analogue of \cite[Theorem~3.3]{kume2024variable} under Assumptions~\ref{assumption 1} and \ref{assumption 2}.
\end{remark}
\begin{proof}
Using \eqref{yh_optcon}, we see that for each $t\ge 1$,
\begin{align}
      &{\rm dist}(0, \nabla f(x^{t+1})+ \partial g(x^{t+1})+ J_c(x^t)^T\partial h(y^t)) \nonumber \\
    &\le \left\|-\frac{2}{\mu_t}(x^{t+1}-x^t)- \left( \beta_t -\beta_{t-1} \right)J_c(x^t)^T(c(x^t) -y^t) +\nabla f(x^{t+1}) -\nabla f(x^t) \right\|      \nonumber\\
    &\le \frac{2}{\mu_t} \|x^{t+1} -x^t\|+ (\beta_t -\beta_{t-1})\| J_c(x^t)\| \|c(x^t) -y^t\| + \| \nabla f(x^{t+1}) -\nabla f(x^t)\| \nonumber \\
    &\overset{(a)}{\le} \frac{2}{\mu_t} \|x^{t+1} -x^t\|+ \eta_0t^{\delta-1}M_cM_0+ \| \nabla f(x^{t+1}) -\nabla f(x^t)\| \nonumber \\
    &\overset{(b)}{\le} \left(L + \frac{2}{\mu_t}\right) \|x^{t+1} -x^t\|+ \eta_0t^{\delta-1}M_cM_0, \label{dist_dubgrad_2}
\end{align}
where $(a)$ holds because of Assumption~\ref{assumption_beta}(ii), Assumption~\ref{assumption 1}(ii), \eqref{bd_xt-yt} and the definition of $M_0$ in \eqref{M0}, and $(b)$ holds because of Assumption~\ref{assumption 1}(i). 

Next, notice that
\begin{align}
    &\frac{1}{T} \sum_{t=1}^T \frac{1}{t^{2-2\delta}} = \frac{1}{T}\left(1+\sum_{t=2}^T t^{2\delta -2} \right) \le \frac{1}{T} \left(1+ \sum_{t=2}^T \int_{t-1}^t s^{2\delta -2} ds \right) \nonumber \\
    &= \frac{1}{T}\left( 1+ \frac{1- T^{2\delta-1}}{1-2\delta}\right) \le \frac{1}{T}\left( 1+ \frac{1}{1-2\delta}\right) \le \frac{2}{(1-2\delta)T} \le \frac{4}{(1-2\delta)(T+1)}, \label{sum_1/t}
\end{align}
where the first inequality holds because $x\mapsto x^{2\delta -2}$ is decreasing for $x>0$ since $\delta \in (0,1/2)$, and the last inequality holds because $2T \ge T+1$.
Then we can deduce from \eqref{dist_dubgrad_2} that
\begin{align}
    &\frac{1}{T}\sum_{t=1}^T {\rm dist}^2(0, \nabla f(x^{t+1})+ \partial g(x^{t+1})+ J_c(x^t)^T\partial h(y^t))   \nonumber \\
    & \le \frac{1}{T}\sum_{t=1}^T \left(\left(L + \frac{2}{\mu_t}\right) \|x^{t+1} -x^t\|+ \eta_0t^{\delta-1}M_cM_0 \right)^2      \nonumber \\
    &\le \frac{1}{T}\sum_{t=1}^T  2\left(L+\frac{2}{\mu_t}\right)^2\|x^{t+1} -x^t\|^2 + \frac{1}{T}\sum_{t=1}^T 2\eta_0^2t^{2\delta-2}M_c^2M_0^2 \nonumber \\
    &= \frac{1}{T}\sum_{t=1}^T2\left(L^2+ \frac{4}{\mu_t}+ \frac{4}{\mu_t^2} \right) \|x^{t+1} -x^t\|^2 +  \frac{1}{T}\sum_{t=1}^T 2\eta_0^2t^{2\delta-2}M_c^2M_0^2 \nonumber \\
    &\overset{(a)}{\le} \frac{1}{T}\sum_{t=1}^T2\left(L^2+ \frac{4}{\mu_t}+ \frac{4}{\mu_t^2} \right) \|x^{t+1} -x^t\|^2  +\frac{8\eta_0^2M_c^2M_0^2}{(1-2\delta)(T+1)},  \label{sum_nlip_dist}
\end{align}
where we used \eqref{sum_1/t} in $(a)$.

Then, it holds that
\begin{align}
    &\min_{1\le t\le T} \left\{{\rm dist}^2(0, \nabla f(x^{t+1})+ \partial g(x^{t+1})+ J_c(x^t)^T\partial h(y^t)) +\|x^{t+1} -x^t\|^2\right\} \nonumber \\
    &\le \frac{1}{T} \sum_{t=1}^T {\rm dist}^2(0, \nabla f(x^{t+1})+ \partial g(x^{t+1})+ J_c(x^t)^T\partial h(y^t)) +\frac{1}{T}\sum_{t=1}^T\|x^{t+1} -x^t\|^2   \nonumber \\
    &\overset{(a)}{\le} \frac{1}{T}\sum_{t=1}^T\left(2L^2 + 1 + \frac{8}{\mu_t}+ \frac{8}{\mu_t^2} \right) \|x^{t+1} -x^t\|^2  + \frac{8\eta_0 ^2M_c^2M_0^2}{(1-2\delta)(T+1)} \nonumber \\
    &\overset{(b)}\le \frac{(8L^2+4)\mu_{\max}M_1}{T+1}+ \frac{(8L^2+4)\mu_{\max}M_0^2\gamma_0 }{(T+1)^{1-\delta}} + \frac{32M_1}{T+1}+ \frac{32M_0^2\gamma_0}{(T+1)^{1-\delta}} \nonumber \\
    &\ \ \ \ + \frac{32\lambda_5M_1}{(T+1)^{1-\delta}}+ \frac{32\lambda_5M_0^2\gamma_0}{(T+1)^{1-2\delta}}+\frac{8\eta_0^2M_c^2M_0^2}{(1-2\delta)(T+1)} \nonumber \\
    &= \frac{\lambda_6}{T+1}+ \frac{\lambda_7}{(T+1)^{1-\delta}}+ \frac{\lambda_8}{(T+1)^{1-2\delta}}, \nonumber
\end{align}
where $(a)$ follows from \eqref{sum_nlip_dist}, $(b)$ holds because of Proposition~\ref{theo_xk+1-xk} and the last equality holds in view of the definitions of $\lambda_6$, $\lambda_7$ and $\lambda_8$ in \eqref{lam_6}, \eqref{lam_7} and \eqref{lam_8}, respectively.
This proves \eqref{min_sub3} and hence establishes item (i).

We now prove item (ii). We first fix a subsequence $\{b_{T_k}\} \subseteq \{b_T\}$ satisfying $b_{T_{k}} \le b_{T_k-1}$ and $T_k >1$. From Remark~\ref{construct_sub_sq} (where we let $a_t := \mu_t^{-2}\|x^{t+1} - x^t\|^2$) and \eqref{conclu_nonlip_1}, we obtain
\begin{align}
    & \frac{1}{\mu_{T_k}^2} \|x^{T_k+1} -x^{T_k}\|^2 \le \frac{4\lambda_5M_1}{T_k}+ \frac{4\lambda_5M_0^2\gamma_0^2}{T_k^{1-\delta}}. \label{tem3}
\end{align}
Hence, $\mu^{-1}_{T_k}(x^{T_{k}+1}- x^{T_k})\to 0$. Since $\mu_t \le \mu_{\max}$ for all $t$, we also deduce that $x^{T_{k}+1}- x^{T_k}\to 0$. 

We next show that $\lim_{k\to\infty}\|c(x^{T_k}) - y^{T_k}\| = 0$. Since $\dom g$ is bounded by Assumption~\ref{assumption 2} and \eqref{bd_xt-yt} holds, we see that $\{(x^{T_k},y^{T_k})\}$ is bounded. Thus, it suffices to show that for any accumulation point $(\hat{x},\hat{y})$ of $\{(x^{T_k},y^{T_k})\}$, we have $c(\hat{x}) = \hat{y}$.

To this end, fix any accumulation point $(\hat{x},\hat{y})$ of $\{(x^{T_k},y^{T_k})\}$. Note that we have $0 < \frac{1}{\mu_t\beta_t} \le \frac{L\rho^{-1}}{\beta_t}+ \rho^{-1} (L_cM_0+ M_c^2)$ by \eqref{bd_1/mu_t}. Then we see that the sequence $\left\{\left(x^{T_k}, y^{T_k}, 1/(\mu_{T_k}\beta_{T_k})\right)\right\}$ is bounded. By passing to a further subsequence if necessary, we may assume that $\left(x^{T_k}, y^{T_k}, 1/(\mu_{T_k}\beta_{T_k})\right)\to (\hat x, \hat y, \hat \tau)$ for some $\hat \tau\ge 0$. 

Recall from \eqref{get_xk_line} that for all $t \ge 0$,
\begin{align}
    \!\!\!x^{t+1}\in \Argmin_{x\in \mathbb{R}^n} \left\{\left\langle  \frac{1}{\beta_t}\nabla f(x^t) + J_c(x^t)^T(c(x^t) -y^t), x\right\rangle + \frac{1}{\beta_t\mu_t}\|x - x^t\|^2  + \frac{1}{\beta_t}g(x)\right\}, \label{xt_proof}
\end{align}
and from \eqref{get_uk_line} that for all $t \ge 1$,
\begin{align}\label{yt_proof}
    y^t &\in  \Argmin_{y\in \mathbb{R}^m} \left\{\frac{1}{\beta_{t-1}} h(y) + \frac{1}{2}\|c(x^t) -y\|^2\right\}.
\end{align}
Now, from \eqref{xt_proof}, we have for any $x\in \dom g$ that
\begin{align*}
&\left\langle  \frac{1}{\beta_{T_k}}\nabla f(x^{T_k}) + J_c(x^{T_k})^T(c(x^{T_k}) -y^{T_k}), x^{T_k+1}\right\rangle + \frac{1}{\beta_{T_k}\mu_{T_k}}\|x^{T_k+1} - x^{T_k}\|^2  + \frac{1}{\beta_{T_k}}g(x^{T_k+1})\\
&\le \left\langle  \frac{1}{\beta_{T_k}}\nabla f(x^{T_k}) + J_c(x^{T_k})^T(c(x^{T_k}) -y^{T_k}), x\right\rangle + \frac{1}{\beta_{T_k}\mu_{T_k}}\|x - x^{T_k}\|^2  + \frac{1}{\beta_{T_k}}g(x).
\end{align*}
Since $x^{T_{k}+1}- x^{T_k}\to 0$ and $\dom g$ is closed (see Assumption~\ref{assumption 1}(iii)), we know that $\lim_{k\to \infty}x^{T_{k}+1}=\lim_{k\to\infty} x^{T_k} = \hat x\in \dom g$. Passing to the limit on both sides of the above inequality and noting that $\beta_{T_k}\to \infty$ and invoking the continuity of $g$ on its closed domain, we deduce that
\begin{align*}
\left\langle J_c(\hat{x})^T(c(\hat{x}) -\hat{y}), \hat{x}\right\rangle \le \left\langle J_c(\hat{x})^T(c(\hat{x}) -\hat{y}), x\right\rangle + \hat\tau\|x - \hat{x}\|^2.
\end{align*}
Since this inequality is true for any $x\in \dom g$, we conclude that
\begin{align}
    \hat{x} &\in \Argmin_{x\in \dom g} \left\{\left\langle J_c(\hat{x})^T(c(\hat{x}) -\hat{y}), x  \right\rangle+ \hat\tau\|x -\hat{x}\|^2\right\}.\label{LO_x}
\end{align}
Using \eqref{yt_proof} and the closedness of $\dom h$ (see Assumption~\ref{assumption 1}(iii)), we deduce similarly that
\begin{align}\label{proj_u*}
     \hat{y} &\in \Argmin_{y\in \dom h} \left\{\frac{1}{2}\|c(\hat{x}) -y\|^2\right\} = {\rm Proj}_{\dom h}(c(\hat{x})).
\end{align}

From \eqref{LO_x}, we obtain
\begin{align*}
    0 \in J_c(\hat{x})^T(c(\hat{x}) -\hat{y}) +\partial\delta_{\dom g}(\hat{x}) &= J_c(\hat{x})^T(c(\hat{x}) -{\rm Proj}_{\dom h}(c(\hat{x})))+\partial \delta_{\dom g}(\hat{x})\\
    & = \partial\left( \frac12{\rm dist}(c(\cdot),\dom h)^2 + \delta_{\dom g}(\cdot)\right)(\hat{x}),
\end{align*}
where the first equality follows from \eqref{proj_u*} and the fact that $\dom h$ is closed and convex, and the second equality follows from \cite[Exercise~8.8(c)]{rockafellar2009variational}.
The above display shows that $\hat{x}$ is a stationary point of \eqref{pro_for_proof}. Hence, $\hat{x}$ is a global minimizer of \eqref{pro_for_proof} thanks to Assumption~\ref{assumption 2}.
Finally, in view of Assumption~\ref{assumption 1}(iii), we see that the optimal value of \eqref{pro_for_proof} is 0. Consequently, we have
\begin{equation}\label{cxy_equals_0}
    \| c(\hat{x}) - \hat{y}\| = {\rm dist}(c(\hat{x}), \dom h) = 0,
\end{equation}
where the first equality follows from \eqref{proj_u*}. Thus, we have shown that $c(\hat{x}) = \hat{y}$ for any accumulation point $(\hat{x},\hat{y})$ of the bounded sequence $\{(x^{T_k},y^{T_k})\}$. Consequently, $\lim_{k\to\infty}\|c(x^{T_k}) - y^{T_k}\| = 0$.

Next, suppose that $x^*$ is an accumulation point of $\{x^{T_k}\}$. Then using $\lim_{k\to\infty}\|c(x^{T_k}) - y^{T_k}\| = 0$ together with the fact that $\{y^{T_k}\}\subseteq \dom h$ and the closedness of $\dom h$, we see that $c(x^*)\in \dom h$. 

Finally, assume further that \eqref{CQ} holds at $x^*$. By passing to a further subsequence if necessary, we may assume that the bounded sequence $\big(x^{T_k}, y^{T_k}, 1/(\mu_{T_k}\beta_{T_k})\big)\to (x^*, y^*, \tau^*)$ for some $y^*\in \dom h$ and $\tau^*\ge 0$. Repeating the argument from \eqref{xt_proof} to \eqref{cxy_equals_0} with $(x^*, y^*, \tau^*)$ in place of $(\hat x, \hat y, \hat \tau)$ shows that $c(x^*) = y^*$. In view of this and \eqref{tem3}, we can now invoke Theorem~\ref{convergence} to show that \eqref{optcon3} holds. 
\end{proof}


\section{Numerical experiments}\label{sec5}

We perform numerical experiments to study the performance of SDCAM$_{\mathds{1}\ell}$. We focus on the settings of Sections~\ref{sec:nonlip} and \ref{sec:43}, which are beyond the reach of the existing literature of single-loop algorithms. All the codes are written in Python 3.13 on a 64 bit PC with AMD Ryzen 7 7435H (16 CPUs), ~3.1GHz and 16GB of RAM.

\subsection{Non-Lipschitz sparse MLP regression}
In this section, we perform numerical experiments for SDCAM$_{\mathds{1}\ell}$ on an instance of Example~\ref{MLP}.

{\bf Data generation.} We use the MNIST handwritten digits dataset \cite{lecun1998gradient}. The original dataset consists of 60000 training images and 10000 test images, each of size $28 \times 28$ pixels. We preprocess the data as follows. First, we flatten each image into a vector of dimension $n_0 = 784$. Then, we rescale the pixel values to the range $[0,1]$ by dividing by 255. For the target values, we use the digit labels ($0-9$) and normalize them to $[-1, 1]$.\footnote{Precisely, we use the mapping $x\mapsto (x-4.5)/4.5$.} We randomly select $m = 1000$ samples from the training set for our experiments. We consider a 3-layer MLP with $(n_0, n_1, n_2) = (784, 128, 64)$. The activation functions for the hidden layers are chosen to be the hyperbolic tangent function ($\sigma_{\tanh}(u) = \frac{e^u - e^{-u}}{e^u+ e^{-u}}$). We let $p = 0.5$ and $\lambda = 0.05$ in \eqref{MLP_model}. 

{\bf Algorithm settings.} We choose $\mu_{\max} = 10^7$, $\mu_{-1} = 0.01$, $\rho = 0.5$ and $\eta = 2$. We generate an $\tilde x$ via Xavier initialization \cite{glorot2010understanding} and project it onto the $C$ defined in \eqref{def_C} to generate $x^0$ and let $y^0 = 0  $ in Algorithm~\ref{alg1}. We let $\delta = 0.5$ and then, at each iteration $t$, let $\beta_t = \beta_0(t+1)^{\delta}$ for some $\beta_0 > 0$: our numerical experiment aims at studying the numerical performance of SDCAM$_{\mathds{1}\ell}$ under different values of $\beta_0$ that will be specified below. We terminate the algorithm when the number of iterations exceeds 3000.

{\bf Numerical results.} We generate an instance described above and compare the performance of SDCAM$_{\mathds{1}\ell}$ with $\beta_0\in \{5\times 10^{-6}, 10^{-5}, 1.5\times 10^{-5}\}$ in Figure~\ref{fig:MLP_results}. We see that a smaller $\beta_0$ leads to a better objective value and a smaller $\|x^{t+1}-x^t\|/\mu_t$ at termination: recall that the latter quantity has been instrumental in our complexity analysis, see Proposition~\ref{theo_xk+1-xk_2}.

\begin{figure}[htbp]
    \centering
    \begin{subfigure}[b]{0.49\linewidth}
        \centering
        \includegraphics[width=\linewidth]{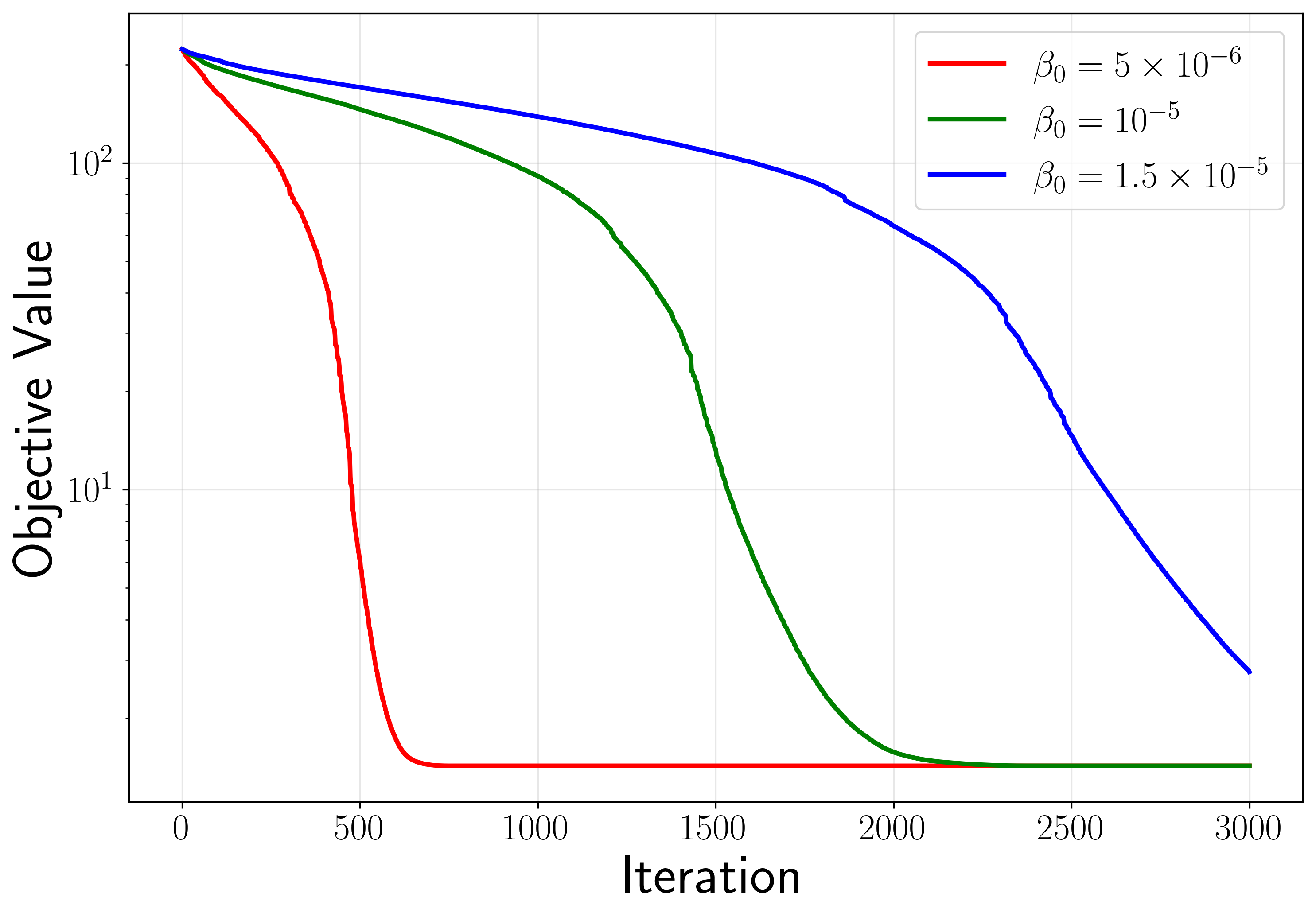}
        \label{fig:MLP_obj_func}
    \end{subfigure}
    \hfill
    \begin{subfigure}[b]{0.49\linewidth}
        \centering
        \includegraphics[width=\linewidth]{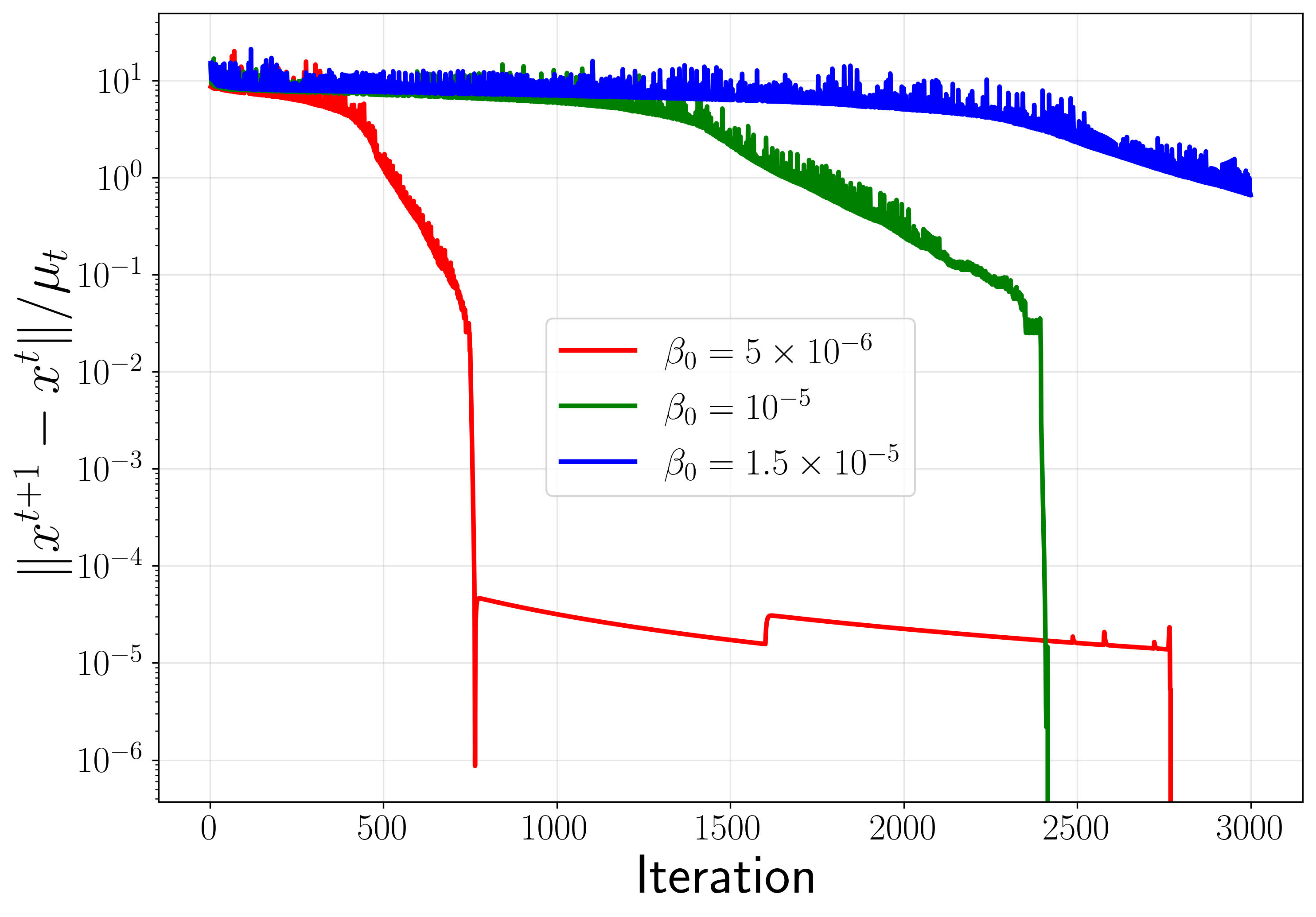}
        \label{fig:MLP_cx-y}
    \end{subfigure}
    
    \caption{Numerical performances of SDCAM$_{\mathds{1}\ell}$ with $\beta_0\in \{5\times 10^{-6}, 10^{-5}, 1.5\times 10^{-5}\}$ for sparse MLP regression.}
    \label{fig:MLP_results}
\end{figure}

\subsection{A penalized QCQP}
In this section, we perform numerical experiments for SDCAM$_{\mathds{1}\ell}$ on an instance of Example~\ref{QCQP}. 

{\bf Data generation.}  We let $Q_0$ be the identity matrix, and generate $b_0$ with standard Gaussian entries and scale it by a factor of ${\rm scale_0} = 5$. Specifically, $b_0 \sim {\rm scale_0}\cdot {\cal N}(0,I_n)$. For each $i = 1,\ldots,m$, we set $Q_i = U_iD_iU_i^T$, where $U_i$ is an $n \times n$ random orthogonal matrix\footnote{Specifically, we first generate a matrix $W_i$ with standard Gaussian entries and then perform a QR decomposition on $W_i$ to obtain the orthogonal matrix $U_i$.} and $D_i$ is a diagonal matrix whose diagonal elements are uniformly distributed in $[0,5]$. In addition, we let $\alpha = 0.05$, and for each $i =1, \ldots, m$, we set $b_i = 0$ and $r_i =- \frac{1}{4} \bar{x}^T Q_i\bar{x}$, where $\bar{x} \in \Argmin~ \frac{1}{2}\|x+b_0\|^2+ 0.05\|x\|_p^p$.\footnote{We minimize $x\mapsto \frac{1}{2}\|x+b_0\|^2+ 0.05\|x\|_p^p$ approximately via a root-finding scheme (see, e.g., \cite{GZLHY13}) and take the approximate minimizer obtained as $\bar x$.} Finally, we set $r = \|\bar{x}\|_{\infty}$ and $p = 0.8$.

{\bf Algorithm settings.} We choose $\mu_{\max} = 10^7$, $\mu_{-1} =1$, $\rho = 0.8$ and $\eta =1.2$ in Algorithm~\ref{alg1}. In addition, we let $\delta = 0.3$ and then, at each iteration $t$, let $\beta_t = \beta_0(t+1)^{\delta}$ for some $\beta_0 > 0$: our numerical experiment aims at studying the numerical performances of SDCAM$_{\mathds{1}\ell}$ under different values of $\beta_0$ that will be specified below.

We solve subproblem \eqref{get_xk_line} approximately through a root-finding scheme (see, e.g., \cite{GZLHY13}). We choose $x^0 = -b_0$ and $y^0 =0$. We terminate the algorithm when the number of iterations exceeds 3000.

{\bf Numerical results.} We consider $(n,m) = (1000, 100)$ and generate an instance described above. We compare the performances of SDCAM$_{\mathds{1}\ell}$ with $\beta_0\in \left\{10^{-4}, 10^{-2},1 \right\}$ in Figure~\ref{fig:numerical_results}. We see that a larger $\beta_0$ leads to a smaller feasibility violation at termination, while a smaller $\beta_0$ leads to a smaller $\|x^{t+1}-x^t\|/\mu_t$ at termination: recall that the latter quantity has been instrumental in our complexity analysis, see Proposition~\ref{theo_xk+1-xk}.
\begin{figure}[!htbp]
    \centering
    \begin{subfigure}{0.49\textwidth}
        \centering
        \includegraphics[width=\linewidth]{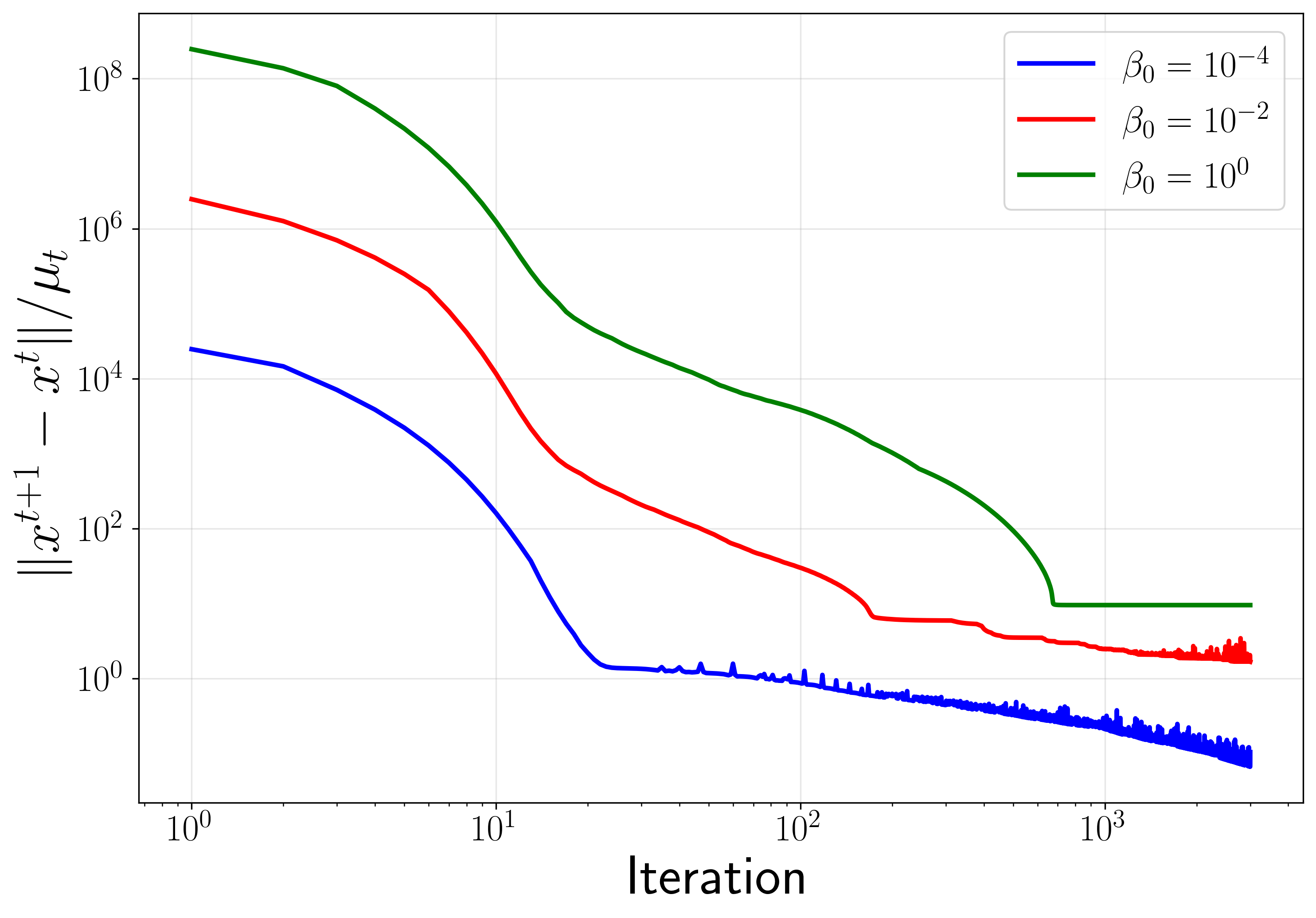}
    \end{subfigure}
    \hfill
    \begin{subfigure}{0.49\textwidth}
        \centering
        \includegraphics[width=\linewidth]{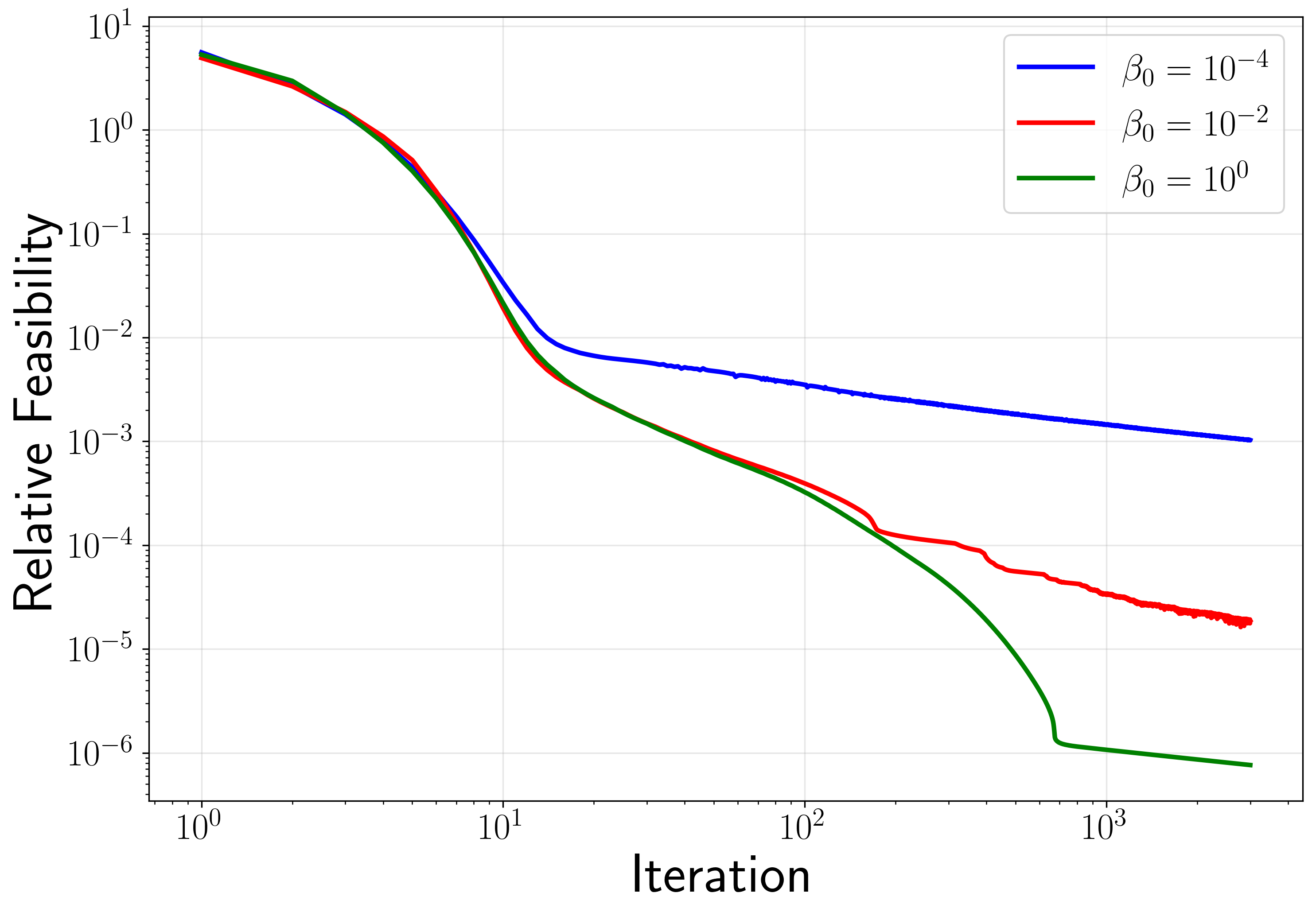} 
    \end{subfigure}
     \caption{\small
        Numerical performances of SDCAM$_{\mathds{1}\ell}$ with $\beta_0\in \{10^{-4}, 10^{-2},1 \}$ for a penalied QCQP. Here, in the second figure, the relative feasibility at the $t$-th iteration is defined as $\| \max\{c(x^t), 0\} / \max\{r_{\rm ref}, 1\}\|$, where $r_{\rm ref} = \begin{bmatrix}|r_1|&\cdots &|r_m|\end{bmatrix}^T$ and the $\max$ and division are performed element-wise.
     }
    \label{fig:numerical_results} 
\end{figure}

\vspace{0.2 cm}

\noindent {\bf Acknowledgements.} The second author is supported in part by JSPS Grant-in-Aid for Research Activity Start-up 	24K23853. The third author is supported in part by the Hong Kong Research Grants Council PolyU 15300423. The fourth author is supported in part by JSPS Grant-in-Aid for Scientific Research (B) JP23H03351.

\noindent {\bf Data availability.} The codes for generating the random data and implementing the
algorithms in the numerical section are available at \url{https://github.com/ahhheee2023/SDCAM_1L}.

\noindent {\bf Competing interests.} The authors declare no competing interests.

\end{document}